\documentclass[reqno]{amsart}
\usepackage{hyperref}
\usepackage{amsfonts}
\usepackage{amsthm,amsmath}
\usepackage[integrals]{wasysym}
\usepackage{enumerate}
\usepackage{fancyhdr}
\usepackage{amssymb}
\usepackage{chemarrow}
\usepackage{tikz}
\usepackage{mathtools}
\usepackage{mathrsfs}

\numberwithin{equation}{section}
\newtheorem{theorem}{Theorem}[section]
\newtheorem{lemma}[theorem]{Lemma}
\newtheorem{definition}[theorem]{Definition}
\newtheorem{proposition}[theorem]{Proposition}
\theoremstyle{remark}
\newtheorem{remark}{Remark}
\allowdisplaybreaks

\newcommand{\with}{\quad\hbox{with}\quad}
\newcommand{\andf}{\quad\hbox{and}\quad}

\def\wt{\widetilde}

\def\cB{{\mathcal B}}
\def\cC{{\mathcal C}}
\def\cD{{\mathcal D}}

\def\cL{{\mathcal L}}

\def\cP{{\mathcal P}}

\def\cS{{\mathcal S}}

\def\du{\delta\!u}
\def\dv{\delta\!v}
\def\dB{\delta\!B}
\def\dJ{\delta\!J}
\def\dU{\delta\!U}

\def\div{ \hbox{\rm div}\,  }
\def\curl{ \hbox{\rm curl}\,  }

\def\N{{\mathbb N}}
\def\R{{\mathbb R}}

\def\Z{{\mathbb Z}}

\def\ddj{\dot\Delta_j}
\def\eps{\varepsilon}

\begin{document}
\title[\hfilneg \hfil ]
{On the well-posedness of  the Hall-magnetohydrodynamics system in critical spaces}

\author{RAPHA\"{E}L DANCHIN AND JIN TAN}


\begin{abstract}
 We investigate  the existence and uniqueness issues of  the 3D incompressible Hall-magnetohydrodynamic system supplemented  with initial
 velocity $u_0$ and magnetic field $B_0$   in critical regularity spaces. 

In the case where  $u_0,$ $B_0$ and the current $J_0:=\nabla\times B_0$  belong to the homogeneous
Besov space $\dot B^{\frac 3p-1}_{p,1},$   $\:1\leq p<\infty,$
and are small enough, we establish a global result and  the conservation of higher regularity.
If  the viscosity is equal to the  magnetic resistivity,  then
we obtain the global well-posedness provided  $u_0,$ $B_0$ and $J_0$ 
are small enough  in 
the \emph{larger} Besov space  $\dot B^{\frac12}_{2,r},$  $r\geq1.$
If   $r=1,$ then we  also establish the local existence for large data, 
and exhibit  continuation criteria for solutions with critical regularity.
 
 Our results rely on an extended formulation of  
the Hall-MHD system, that  has some   similarities with the  incompressible 
Navier-Stokes  equations. 
\end{abstract}

\maketitle
\section{Introduction}
We are concerned with 
the following three dimensional incompressible resistive and viscous Hall-magnetohydrodynamics system (Hall-MHD):
\begin{align}
&\partial_t{\mathnormal u}+\div(\mathnormal u\otimes\mathnormal u)+\nabla \mathnormal P=(\nabla\times \mathnormal B)\times \mathnormal B+\mu\Delta \mathnormal u\label{1.1},\\
&\div\mathnormal u=0\label{1.2},\\
&\partial_t{\mathnormal B}=\nabla\times((\mathnormal u -h\nabla\times\mathnormal B) \times \mathnormal B)+\nu\Delta \mathnormal B\label{1.3}, 
\end{align}
supplemented with the initial conditions
\begin{equation}
(\mathnormal u(0,\mathnormal x), \mathnormal B(0,\mathnormal x))=(\mathnormal u_{0}(\mathnormal x), \mathnormal B_{0}(\mathnormal x)),\quad
x\in\R^3.\label{1.4}
\end{equation}
The unknown vector-fields   $\mathnormal u=\mathnormal u(t,x) $ and  
$\mathnormal B=\mathnormal B(t,x),$ and scalar function $\mathnormal P=\mathnormal P(t,x)$ 
with $t\geq0$ and $x\in\R^3$
represent the velocity field, the magnetic field and the scalar pressure, respectively.  The parameters $\mu$ and $\nu$ are the fluid viscosity and the magnetic resistivity, while the dimensionless number $h$ measures   the 
magnitude  of the Hall  effect compared to  the typical length scale of the fluid. 
In accordance with \eqref{1.2},  we assume that $\div u_0=0$ 
and, for physical consistency,  since a magnetic field has to be 
divergence free,  we suppose  that $\div B_0=0,$ too,
a property that is conserved  through the evolution. 
 \medbreak
The above system is used to model the evolution of electrically conducting fluids such as 
plasmas or electrolytes (then, $u$ represents the ion velocity), and takes into account 
 the fact that in a moving conductive fluid,  the magnetic field can induce currents which, in turn, polarize the fluid and change the magnetic field. 
 That phenomenon which is neglected in the classical MHD equations, 
is represented by  the \emph{Hall electric field}
$E_H:=h J\times B$ where the current  $J$ is defined by $J:=\nabla\times B.$
  Hall term plays an important role in magnetic reconnection, as observed in e.g. 
  plasmas, star formation, solar flares, neutron stars or geo-dynamo (for  more explanation on the physical background of the Hall-MHD system, one can refer to  \cite{Al11,Ba01,Fo91,Hu03,Sh97,Wa04}).
\medbreak
Despite its physical relevance,  the Hall-MHD system has been considered 
only rather recently in mathematics, following the work by Acheritogaray, Degond, Frouvelle and Liu in  \cite{Ac11} 
where the Hall-MHD system was formally derived both   from a two fluids system and  from a  kinetic model. Then, in \cite{Ch13}, Chae, Degond and  Liu showed the global existence of weak solutions as well as the local well-posedness for  initial data $u_0$ and $B_0$ in Sobolev spaces $H^{s}$ with $s>5/2$.  Weak solutions have been further investigated 
by  Dumas and Sueur in \cite{Ds14} both for the Maxwell-Landau-Lifshitz system and for the Hall-MHD system. In \cite{Ch14,Ye15}, Serrin type continuation criteria for smooth solutions and the global existence of strong solutions  emanating from  small initial data have been obtained. 
In \cite{Z16}, it has been observed that the possible blow-up of smooth solutions  
may be controlled \emph{in terms of the velocity only.}
More  well-posedness  results of strong solutions for  
less regular data in Sobolev  or Besov spaces, have been 
established in  \cite{Be16},  \cite{Yt18} and \cite{Wz15}, and the convergence
to the  MHD system with no Hall-term for $h\to0$ has been addressed in \cite{WZ19}. 
Examples of smooth data with arbitrarily large  $L^\infty$ norms
giving rise to  global unique solutions have been exhibited very recently in \cite{Li19}. 
\smallbreak
Our main goal  here is to establish  the well-posedness  of the Hall-MHD system  with initial data in \emph{critical spaces.} 
In contrast with the classical MHD system (that corresponds to $h=0$) however,
the system under consideration does not have any scaling 
invariance owing to the coexistence of the Hall term
 in \eqref{1.3} and of the Lorentz force in \eqref{1.1}, 
 and we  have to explain what we mean by critical regularity. 
\smallbreak
On the one hand, 
for  $B\equiv0,$ the Hall-MHD system reduces to  the incompressible Navier-Stokes: 
$$\left\{
\begin{aligned}
&\partial_t{\mathnormal u}+(\mathnormal u\cdot\nabla)\mathnormal u+\nabla \mathnormal P=\mu\Delta \mathnormal u,\\
&\div\mathnormal u=0,\\
&u|_{t=0}=u_{0},
\end{aligned}
\right.\leqno(NS)$$
which are invariant for all  $\lambda>0$ by the rescaling 
\begin{equation}\label{eq:solutionNS}
u(t,x)\leadsto  \lambda u({\lambda^{2}t, \lambda x})\andf P(t,x)\leadsto  
\lambda^2P({\lambda^{2}t, \lambda x})
\end{equation}
provided the initial velocity  $u_0$ is rescaled according to 
\begin{equation}\label{eq:dataNS}
u_0(x)\leadsto \lambda u_0(\lambda x).
\end{equation}
On the other hand,   if the fluid velocity  in \eqref{1.3} is $0,$ 
then we get the following \emph{Hall equation} for $B$:
$$\left\{\begin{aligned}
&\partial_t{\mathnormal B}+h\nabla\times((\nabla\times\mathnormal B)\times \mathnormal B)
=\nu\Delta \mathnormal B,\\
&B|_{t=0}=B_0,\end{aligned}\right.\leqno(Hall)$$
which is invariant by the rescaling 
\begin{equation}\label{eq:solutionHall}
B(t,x)\leadsto  B({\lambda^{2}t, \lambda x})
\end{equation}
provided the data $B_0$ is rescaled according to 
\begin{equation}\label{eq:dataHall}
B_0(x)\leadsto B_0(\lambda x).
\end{equation}
Therefore,  if $h>0$ and if we neglect the  Lorentz force in \eqref{1.1}, 
 then  it is natural to work at the same level of regularity for $u$ and $\nabla B,$
 while for $h=0$ (the classical MHD system)
$u$ and $B$ have the same scaling invariance. 
 \medbreak
The way to reconcile the two viewpoints is to 
 look at  the current function  $J=\nabla\times B$ as an additional unknown. 
  Now, owing to the vector identity
\begin{equation}
\nabla\times(\nabla\times v) + \Delta v = \nabla{\rm{div}}~v\label{1.600}
\end{equation}
and since $B$ is divergence free,  we have  
$\Delta  B=- \nabla\times J,$
whence
$$
 B={\rm{curl}}^{-1}J{:=}(-\Delta)^{-1}\nabla\times J,$$
where the $-1$ order homogeneous Fourier multiplier ${\rm{curl}}^{-1}$ is defined 
 on the Fourier side by 
\begin{equation}\label{eq:curl-1}\mathcal{F}({\rm{curl}}^{-1}J)(\xi):=\frac{i\xi\times\widehat{J}(\xi)}{|\xi|^2}\cdotp\end{equation}
With that notation, one gets the following \emph{extended Hall-MHD system}:
\begin{equation}\label{eq:HMHD-ext}
\left\{\begin{aligned}
&\partial_t{\mathnormal u}+\div(\mathnormal u\otimes\mathnormal u)-\mu\Delta \mathnormal u+\nabla \mathnormal P=(\nabla\times \mathnormal B)\times \mathnormal B,\\
&\div\mathnormal u=0,\\
&\partial_t{\mathnormal B}-\nabla\times((\mathnormal u -h J) \times \mathnormal B)-\nu\Delta \mathnormal B=0,\\
 &\partial_t{\mathnormal J}-\nabla\times\bigl(\nabla\times((\mathnormal u -h J) \times {\rm{curl}}^{-1}\mathnormal J)\bigr)-\nu\Delta \mathnormal J=0.
\end{aligned}\right.\end{equation}
In contrast with the original Hall-MHD system, the above extended system
 has  a scaling invariance (the same as  for the incompressible Navier-Stokes equations).   Studying  whether
the Hall-MHD system written in terms of $(u,B,J)$  is well-posed 
in  the same functional spaces as the  velocity in $(NS),$ and if similar blow-up criteria and qualitative behavior may be established is the main aim of the paper.
\medbreak\noindent{\bf Notation.} 
We end this introductory part presenting  a few notations. 
We denote by $C$ harmless positive `constants' that  may change from one line to the other, and we sometimes write $A\lesssim B$ instead of $A\leq C B.$  Likewise,    $A\sim B$ means that  $C_1 B\leq A\leq C_2 B$ with absolute constants $C_1$, $C_2$. For $X$ a Banach space, $p\in[1, \infty]$ and $T\in(0,\infty]$, the notation $L^p(0, T; X)$ or $L^p_T(X)$ designates the set of measurable functions $f: [0, T]\to X$ with $t\mapsto\|f(t)\|_X$ in $L^p(0, T)$, endowed with the norm $\|\cdot\|_{L^p_{T}(X)} :=\|\|\cdot\|_X\|_{L^p(0, T)}.$ For any interval $I$ of $\R,$ we agree that $\mathcal C(I; X)$ (resp. ${\mathcal C}_b(I;X)$) denotes the set of continuous 
(resp. continuous and bounded) functions from $I$ to $X$.  We  keep the same notation for functions with several components.
\medbreak\noindent{\bf Acknowledgments.}  
 The authors are  indebted  to the anonymous referees 
 for their wise  suggestions that  contributed to improve the final version of this work. 
 The first author is partially  supported  by the ANR project INFAMIE (ANR-15-CE40-0011).
 The second author has been partly funded by the B\'ezout Labex, funded by ANR, 
 reference ANR-10-LABX-58.


\section{Main results}

It is by now classical that  the incompressible Navier-Stokes equations 
 are well-posed (locally for large data or globally for small data) 
 in all homogeneous Besov spaces $\dot B^{\frac3p-1}_{p,r}$ with 
 $1\leq p<\infty$ and $1\leq r\leq \infty,$ see \cite{cannone,Ch99}.
Similar results have been obtained  for 
 the standard incompressible MHD system (that is with no Hall term)
 by C. Miao and B. Yuan in \cite{Mi09}. 
 According to  the scaling considerations of the above paragraph, 
 it is  natural to look for  similar results  for the Hall-MHD system 
written in its extended formulation \eqref{eq:HMHD-ext}.

\medbreak
Let us first  consider the case of positive general coefficients $\mu,$ $\nu$ and $h,$ 
for  data $(u_0,B_0,J_0)$  in   $\dot B^{\frac3p-1}_{p,1}.$ 
Following the paper \cite{Ch99}  by  J.-Y. Chemin 
dedicated to the  incompressible Navier-Stokes equations, 
we introduce for $T>0,$ the   space 
$$E_{p}(T){:=}\Big\{z\in \cC([0, T], \dot B^{\frac{3}{p}-1}_{p, 1}),~\nabla_x^2 z\in  L^1(0, T; \dot B^{\frac{3}{p}-1}_{p, 1})\andf \div_{\!x}\,z=0\Big\}
$$
and its global version  $E_{p}$ (with $z\in\cC_b(\R_+;\dot B^{\frac{3}{p}-1}_{p, 1})$)  if  $T=+\infty$.
\medbreak
Our first result states the global well-posedness of the Hall-MHD system for small data in $\dot B^{\frac3p-1}_{p,1},$
 and conservation of higher order Sobolev regularity for \emph{any positive coefficients $\mu,$ $\nu$ and $h$}. 
  \begin{theorem}\label{Th_1}
Let $1\leq p<\infty$ and  $(u_0, B_0)\in{\dot B^{\frac{3}{p}-1}_{p, 1}}$ 
with $\div u_0=\div B_0=0$ 
and $J_0:=\nabla\times B_0\in{\dot B^{\frac{3}{p}-1}_{p, 1}}.$  There exists a constant $c>0$  
depending only on $p$ and $\mu/\nu$ such that, if 
\begin{equation}\label{eq:smallness} 
\|u_0\|_{\dot B^{\frac{3}{p}-1}_{p, 1}}+\|B_0\|_{\dot B^{\frac{3}{p}-1}_{p, 1}}+h\|J_0\|_{\dot B^{\frac{3}{p}-1}_{p, 1}}<c\mu,
\end{equation}
 then the Cauchy problem \eqref{1.1}-\eqref{1.4} has a unique global solution $(u, B)\in E_{p},$  with $J:=\nabla\times B\in E_{p}.$ Furthermore, 
\begin{equation}
\|u\|_{E_{p}}+\|B\|_{E_{p}}+h\|J\|_{E_{p}}<2c\mu\label{1.1200}.
\end{equation}
If, in addition, $u_0\in H^s$ and  $B_0\in H^r$ with
\begin{equation}\label{cond:reg}\frac{3}{p}-1<s\leq r\andf\frac{3}{p}<r\leq 1+s, \end{equation}
then
$(u,B)\in\cC_b(\R_+;H^s\times H^r),$ $\nabla u\in L^2(\R_+;H^s)$ and~$\nabla B\in L^2(\R_+;H^r)$
and the following energy balance is fulfilled for all $t\geq0$:
\begin{equation}\label{eq:energy}
\|u(t)\|_{L^2}^2+\|B(t)\|_{L^2}^2 
+2\int_0^t\bigl(\mu\|\nabla u\|_{L^2}^2+\nu\|\nabla B\|_{L^2}^2\bigr)\,d\tau=
\|u_0\|_{L^2}^2+\|B_0\|_{L^2}^2.\end{equation} 
Finally, in the case where only $J_0$ fulfills \eqref{eq:smallness}, 
there exists some time $T>0$ such that  \eqref{1.1}-\eqref{1.4}
has a unique local-in-time solution on $[0,T]$ with $(u,B,J)$ in $E_{p}(T),$
and additional Sobolev regularity is preserved. 
\end{theorem}
\begin{remark}   For $h$ going to $0,$ we recover
the same smallness condition as in the paper   \cite{Mi09} dedicated
to the MHD system with no Hall-term.
\end{remark}
\begin{remark}
Whether the smallness condition on $J_0$ may be omitted in the context 
of general critical regularity spaces $\dot B^{\frac3p-1}_{p,1}$ is an open question. 
The difficulty not only comes from Hall term but also from the coupling between $u$ and $B$ through the term $\nabla\times(u\times B).$
For essentially the same reason, we do not know how to solve the system in $\dot B^{\frac3p-1}_{p,r}$
if $r>1,$ unless $p=2$ and $\mu=\nu,$ as we shall see later on. 
\end{remark}
\bigbreak
The key to the proof of Theorem \ref{Th_1} 
is to consider  the extended Hall-MHD system  \eqref{eq:HMHD-ext}, suitably rewritten 
in the form of  a generalized Navier-Stokes system that may be solved by implementing
the  classical fixed point theorem in  the  Banach space $E_p.$   
In order to derive an appropriate formulation of the system, we need
to use  some algebraic identities. 
The first one is  that for any couple of $C^1$  divergence free vector-fields  $v$ and $w$ on $\mathbb R^{3},$ we have
\begin{equation}
w\cdot\nabla v={\rm{div}}(v\otimes w),\ \hbox{ where }\ \bigl({\rm{div}}(v\otimes w)\bigr)^{j}{:=}\sum_{k=1}^{3}\partial_{k}(v^{j}w^{k})\label{1.7}.
\end{equation}
Observe also that 
\begin{equation}
(\nabla\times\mathnormal w)\times\mathnormal w=(\mathnormal w\cdot\nabla)\mathnormal w-\nabla\Bigl(\frac{|\mathnormal w|^{2}}{2}\Bigr)\cdotp
\end{equation}
Hence, setting $Q:= P+|\mathnormal B|^{2}/{2},$ equation  \eqref{1.1} recasts in
\begin{equation}\label{eq:u1}
\partial_t u+{\rm{div}}(u\otimes u)+\nabla Q={\rm{div}}(B\otimes B)+\mu\Delta u.
\end{equation}
After projecting \eqref{eq:u1}  onto the set of divergence free vector fields
by means of the \emph{Leray projector} ~$\mathcal{P}:=\rm{Id}-\nabla(-\Delta)^{-1}{\rm{div}},$ we get
\begin{equation}
\partial_t u-\mu\Delta u=Q_{a}(B, B)-Q_{a}(u, u)\label{1.9},
\end{equation}
where the bilinear form $Q_{a}$ is defined by
\begin{equation*}
Q_{a}(v, w){:=}\frac{1}{2}\mathcal{P}({\rm{div}}(v\otimes w)+{\rm{div}}(w\otimes v)).
\end{equation*}
Next, by using the identity
\begin{equation}
\nabla\times(\mathnormal w\times\mathnormal v)=\mathnormal v\cdot\nabla\mathnormal w-\mathnormal w\cdot\nabla\mathnormal v\label{1.10},
\end{equation}
one  can rewrite Hall term as
\begin{equation*}
\nabla\times(J\times B)=B\cdot\nabla J-J\cdot\nabla B.
\end{equation*}
Hence, combining with \eqref{1.7}, equation \eqref{1.3} recasts in
\begin{equation*}
\partial_t B-\nu\Delta B=Q_{b}(B, h J-u),
\end{equation*}
where
$$Q_{b}(v, w){:=}{\rm{div}}(v\otimes w)-{\rm{div}}(w\otimes v)=w\cdot\nabla v-v\cdot\nabla w,$$
and the equation for $J$ may thus be written
$$\partial_t J-h\Delta J=\nabla\times Q_{b}({\rm{curl}}^{-1}J, h J-u).$$
Altogether,  we conclude that the   extended Hall-MHD system  \eqref{eq:HMHD-ext} recasts in 
\begin{equation} \label{eq:HMHD-ext2} \quad\left\{
\begin{aligned}
 &\partial_t u-\mu\Delta u=Q_{a}(B, B)-Q_{a}(u, u),\\
&\partial_t B-\nu\Delta B=Q_{b}(B, h J-u),\\
&\partial_t J-\nu\Delta J=\nabla\times Q_{b}({\rm{curl}^{-1}}J, h J-u),\\
&(\mathnormal u(0,\mathnormal x), \mathnormal B(0,\mathnormal x), \mathnormal J(0,\mathnormal x))=(\mathnormal u_{0}, \mathnormal B_{0}, J_{0}).
\end{aligned}
\right.\end{equation}
Set $U:=(U_1,U_2,U_3)$ with 
 $U_1:=u,$ $U_2:=B$ and $U_3:=J.$ Then,    \eqref{eq:HMHD-ext2} may be shortened into: 
\begin{equation}\label{eq:U}
 \left\{
\begin{aligned}
 &\partial_t U-\Delta_{\mu,\nu} U=Q(U, U),\\
 &U|_{t=0}=U_{0},
\end{aligned}
\right.\qquad\with \Delta_{\mu,\nu}U:=\left(\begin{array}{c}\mu \Delta u\\ \nu \Delta B\\ \nu \Delta J\end{array}\right)
\end{equation}
and where $Q:\mathbb R^{3}\times\mathbb R^{3}\to\mathbb R^{9}$ is defined by
\begin{equation}\label{eq:Q} 
 Q(V, W){:=}\left(  \begin{array}{c}
         Q_{a}(V_{2}, W_{2})-Q_{a}(V_{1}, W_{1}) \\
         Q_{b}(V_{2}, h W_{3}-W_{1})\\
         \nabla\times Q_{b}({\rm{curl}}^{-1}V_{3}, h W_{3}-W_{1})
          \end{array}
          \right)\cdotp
\end{equation}
The gain of considering the above extended system rather than the initial one is
that it  is semi-linear, while the Hall-MHD system for $(u,B)$ is quasi-linear.
The quadratic terms in the first two lines of \eqref{eq:U} are essentially 
the same as for  the incompressible Navier-Stokes equation. 
Owing to the Hall-MHD term in  the last line however, one has  to go beyond the theory of
the generalized Navier-Stokes equations as presented in e.g. \cite[Chap. 5]{Ba11} 
since the differentiation is \emph{outside} instead of being inside the first variable of $Q_b$
(this actually prevents us from considering 
large $J_0$'s  and  to handle regularity in Besov spaces  $\dot{B}^{\frac{3}{p}-1}_{p, r}$ with $r>1$). 
\medbreak
 In the case\footnote{That is made in most mathematical papers devoted to the Hall-MHD system even though it is not physically motivated.}
  $\mu=\nu,$    there is a cancellation property
 that eliminates  the Hall term when performing an energy method, so that
 one can obtain better results. 
 In particular, one can  prove the local well-posedness for general large data  in 
 $\dot{B}_{2,1}^{\frac{1}{2}}$ and  the global well-posedness for small data in all  
 spaces 
 $\dot{B}_{2,r}^{\frac{1}{2}}$ with  $r\in[1,\infty].$ 
 \medbreak 
 In order to explain where that  cancellation comes from,  let us
   introduce the function $v:=u-h J$ 
 (that may be interpreted as the velocity of an electron, see \cite{Al11} page 5).
 Recall the following 
 vector identities:
$$\nabla(w\cdot z)=(\nabla w)^T z+(\nabla z)^T w
\andf(\nabla w-(\nabla w)^T)z=(\nabla\times w)\times z,$$
where $(\nabla w)_{ij}:=\partial_jw^i$ for $1\leq i,j\leq 3.$
  \medbreak
Hence, combining  with \eqref{1.10} yields
\begin{align}\nabla\times(w\times z)&=z\cdot\nabla w- w\cdot\nabla z\notag\\
&=(\nabla w-(\nabla w)^T)z+(\nabla z-(\nabla z)^T)w-2w\cdot\nabla z+\nabla(w\cdot z)\notag\\
&=(\nabla\times w)\times z+(\nabla\times z)\times w-2w\cdot\nabla z+\nabla(w\cdot z)\label{1.1300}.
\end{align}
  Then, applying Identity \eqref{1.1300} to  the term
$\nabla\times(v\times B),$
equation \eqref{1.3} turns into
\begin{equation*}
\partial_t B-\mu\Delta B=(\nabla\times v)\times B+J\times u-2v\cdot\nabla B+\nabla(v\cdot B).
\end{equation*}
Taking   $h\cdot{\rm{curl}}$  of the above equation, and subtracting  it from \eqref{eq:u1},  we  get
$$\displaylines{
\quad\partial_tv-\mu\Delta v
=B\cdot\nabla B-u\cdot\nabla u-h\nabla\times((\nabla\times v)\times B)
\hfill\cr\hfill+h\nabla\times(u\times J)
+2h\nabla\times(v\cdot\nabla B)-\nabla Q.\quad}$$
Therefore,  since $hJ=v-u,$  in terms of unknowns $(u,B,v),$ the extended 
  Hall-MHD system reads 
\begin{equation}\label{1.100}
\left\{
\begin{aligned}
&\partial_t u-\mu\Delta u=B\cdot\nabla B-u\cdot\nabla u-\nabla Q,\\
&\div u=0,\\
&\partial_t{\mathnormal B}-\mu\Delta B=\nabla\times(v\times B),\\
&\partial_tv-\mu\Delta v=B\cdot\nabla B-u\cdot\nabla u-h\nabla\times((\nabla\times v)\times B)\\
&\hspace{4cm}+\nabla\times(v\times u)+2h\nabla\times(v\cdot\nabla B)-\nabla Q.\\
\end{aligned}
\right.
\end{equation}
 The only quasilinear  term cancels out when performing an energy method, since
\begin{equation}\label{1.1100}
(\nabla\times((\nabla\times v)\times B), v)_{L^2}=0.
\end{equation}
After localization of the system by means of the Littlewood-Paley  spectral 
cut-off operators $\dot \Delta_j$ defined in the Appendix, the above 
identity still holds, up to some lower order commutator term.
This will enable us to prove the following 
local  well-posedness result \emph{for large data} in the critical Besov space $\dot B^{\frac12}_{2,1},$ 
together with blow-up criteria involving critical norms. 
\begin{theorem}\label{Th_2} Assume that $\mu=\nu.$
For any initial data $(u_0, B_0)$ in ${\dot B^{\frac{1}{2}}_{2, 1}}$  with $\div u_0=\div B_0=0$ 
and $J_0:=\nabla\times B_0\in{\dot B^{\frac{1}{2}}_{2, 1}},$ 
there exists a positive time $T$ such that the Cauchy problem \eqref{1.1}-\eqref{1.4} has a unique solution $(u, B)\in E_{2}(T)$ with  $J:=\nabla\times B\in E_{2}(T)$. Moreover, if the maximal time of existence $T^*$ of that solution is finite, then 
\begin{eqnarray}\label{blowup1}
&&\int_0^{T^*}\|(u, B, \nabla B)(t)\|_{L^\infty}^2
\,dt =\infty\\\label{blowup2}
&&\int_0^{T^*}\|(u, B, \nabla B)(t)\|_{\dot B^{\frac{5}{2}}_{2, 1}}\,dt =\infty
\end{eqnarray}
and, for any $\rho\in(2, \infty),$
\begin{equation}\label{blowup3}
\int_0^{T^*}\|(u, B, \nabla B)(t)\|_{\dot B^{\frac{2}{\rho}-1}_{\infty, \infty}}^\rho\,dt =\infty.
\end{equation}
\end{theorem}  
Still for  $\mu=\nu,$   one can prove well-posedness 
in \emph{any} critical space $\dot B^{\frac12}_{2,r}$ with $r\in[1,\infty].$ 
Then,  the components of the solution will belong to the 
following space\footnote{The reader may refer to Definition \ref{d:tilde}
for the definition of `tilde spaces'}:
$$E_{2,r}(T){:=}\Big\{v\in \wt\cC_T(\dot B^{\frac12}_{2, r}),~\nabla^2 v\in  \wt L^1_T(\dot B^{\frac12}_{2, r})\andf\div v=0\Big\},$$
where the letter $T$ is omitted if the time interval is $\R_+.$
\begin{theorem}\label{Th_2bis} Assume that $\mu=\nu.$
Consider  initial data $(u_0, B_0)$ in ${\dot B^{\frac{1}{2}}_{2, r}}$  with $\div u_0=\div B_0=0$ 
and $J_0:=\nabla\times B_0\in{\dot B^{\frac{1}{2}}_{2, r}}$ for some $r\in[1,\infty].$ 
 Then, the following results hold true:
\begin{enumerate}
\item there exists a universal positive constant $c$ such that, if
\begin{equation}\label{eq:smallnessr} 
\|u_0\|_{\dot B^{\frac12}_{2, r}}+\|B_0\|_{\dot B^{\frac12}_{2,r}}+\|u_0-h J_0\|_{\dot B^{\frac12}_{2, r}}<c\mu,
\end{equation}
then the Hall-MHD system has a unique global solution $(u,B)$ with 
$(u,B,J)$ in $E_{2,r}.$
\item If only  $\|u_0-h J_0\|_{\dot B^{\frac12}_{2, r}}<c\mu,$  
then  there exists $T>0$
such that  the Hall-MHD system  has a unique solution $(u,B)$ on $[0,T],$ with $(u,B,J)$ in $E_{2,r}(T).$ 
\end{enumerate}
\end{theorem}
\begin{remark} Observe that for $h=0,$ one recovers the statement of \cite{Mi09}
pertaining to  the classical incompressible MHD system.
\end{remark}
The rest of the paper unfolds as follows. 
The next section is devoted to the proof of Theorem \ref{Th_1}. 
In Section \ref{s:new}, we focus on the case $\mu=\nu$ and 
prove Theorem \ref{Th_2}  by taking advantage of the cancellation 
property pointed out above.  The proof of  Theorem \ref{Th_2bis} is carried out in
Section \ref{s:r}.  
For the reader's convenience, results concerning 
Besov spaces, Littlewood-Paley decomposition and commutator estimates are recalled in Appendix.


\section{Well-posedness in general critical Besov spaces with third index 1}\label{s:small}

The present section is dedicated to proving  Th. \ref{Th_1}.
Before starting, a fundamental observation (that will be also used
in the next sections) is in order:
the triplet $(u,B,P)$ satisfies the Hall-MHD system \eqref{1.1}-\eqref{1.3} with coefficients
$(\mu,\nu,h)$ if and only if  the \emph{rescaled} triplet:
\begin{equation}\label{eq:rescaling}
(\wt u,\wt  B,\wt P)(t,x):=\frac{h}{\mu}\Bigl(u,B, \frac{h}{\mu} P\Bigr)\Bigl(\frac{h^2}\mu t, hx\Bigr)
\end{equation}
satisfies  \eqref{1.1}-\eqref{1.3} with coefficients $(1,1,\nu/\mu).$
\medbreak
Consequently, taking advantage of the scaling invariance of the homogeneous
Besov norms (see Proposition \ref{P_25} ${\rm{(vi)}}$), it is enough 
to prove the statement in the case where the viscosity $\mu$
and the Hall number $h$ are equal to $1.$

For expository purpose, 
we shall assume in addition that the magnetic resistivity $\nu$
is  equal to $1$ (to achieve the general case it is only a matter 
of changing the heat semi-group in the definition of  $\cB$ in \eqref{3.1} below accordingly).
\medbreak
Throughout this section and the following ones, we shall repeatedly  use the fact that, 
as a consequence of Proposition \ref{P_25} ${\rm{(vii)}},$  one has the following
equivalence of norms for all $s\in\R$ and $(p,r)\in[1,+\infty]^2$: 
\begin{equation}\label{eq:equivnorm}
\|\nabla B\|_{\dot{B}_{p, r}^{s}}\sim \|J\|_{\dot{B}_{p, r}^{s}}\andf
\|\nabla B\|_{\dot{H}^{s}}=\|J\|_{\dot{H}^{s}}.\end{equation}

In order to establish the global existence
of a solution of the Hall-MHD system  in the case of small data, we shall first prove 
the corresponding result for the extended system \eqref{eq:U}. It relies  on   the following 
well known corollary of the  fixed point theorem  in
 complete metric  spaces.
 \begin{lemma}\label{Le_31}
Let $(X, \|\cdot\|_{X})$ be a Banach space and $\mathcal B : X\times X\to X,$  a bilinear continuous operator with norm $K$. Then, for all $y\in X$ such that $4K\|y\|_{X}<1$, equation 
$$x=y+\mathcal B(x,x)$$ has a unique solution $x$ in the ball $\mathit B(0, \frac{1}{2K})$. Besides, $x$ satisfies $\|x\|_{X}\leq2\|y\|_{X}$.
\end{lemma}

We shall take for $X$ the set of triplets of (time dependent) divergence free vector-fields 
with components in $E_p$
endowed with the norm  
\begin{equation*}
\|V\|_{X}:~=\|V\|_{{L}^{1}(\dot B^{\frac{3}{p}+1}_{p, 1})}+\|V\|_{{L}^{\infty}(\dot B^{\frac{3}{p}-1}_{p, 1})}.
\end{equation*}
Let $(e^{t\Delta})_{t\geq0}$ denote the heat semi-group defined in \eqref{eq:heat}. 
We set~ $y:~t\mapsto e^{t\Delta}U_{0}$ and define the bilinear functional $\mathcal B$ by the formula
\begin{equation}
\mathcal B(V, W)(t)=\int_{0}^{t}e^{(t-\tau)\Delta}Q(V, W)~d\tau\label{3.1}.
\end{equation}
By virtue of \eqref{2.1}, System \eqref{eq:U}  recasts in 
\begin{equation}
U(t)= y(t)+\mathcal B(U, U)(t)\label{3.2}.
\end{equation}

In order to apply  Lemma \ref{Le_31}, it suffices to show that  $y$ is small 
in $X,$ and that $\mathcal B$ maps $X\times X$ to $X.$ 
The former property holds true if Condition  \eqref{eq:smallness} is fulfilled for a small enough $c>0,$ as  Proposition~\ref{Le_27}  ensures that $y$ belongs to $X$ and that 
$$\|y\|_{X}\leq C\|U_{0}\|_{\dot B^{\frac{3}{p}-1}_{p, 1}}.$$
 
In order to prove the latter property, one can use the fact that, by virtue of Identity \eqref{1.7}, 
Proposition \ref{P_25} (i), (iii), (vii),  and 
Inequality \eqref{prod1},  we have
\begin{align}
\|{\rm{div}}(v\otimes w)\|_{\dot B^{\frac{3}{p}-1}_{p, 1}}&\lesssim\|v\otimes w\|_{\dot B^{\frac{3}{p}}_{p, 1}}\notag\\
&\lesssim\|v\|_{\dot B^{\frac{3}{p}}_{p, 1}}\|w\|_{\dot B^{\frac{3}{p}}_{p, 1}},\label{3.5}
\end{align}
\begin{align}
\|{\rm{div}}(({\rm{curl}}^{-1}v)\otimes w)\|_{\dot B^{\frac{3}{p}}_{p, 1}}&=\|w\cdot\nabla ({\rm{curl}}^{-1}v)\|_{\dot B^{\frac{3}{p}}_{p, 1}}\notag\\
&\lesssim\|\nabla {\rm{curl}}^{-1}v\|_{\dot B^{\frac{3}{p}}_{p, 1}}\|w\|_{\dot B^{\frac{3}{p}}_{p, 1}}\notag\\
&\lesssim\|v\|_{\dot B^{\frac{3}{p}}_{p, 1}}\|w\|_{\dot B^{\frac{3}{p}}_{p, 1}}\notag\\
&\lesssim\|v\|_{\dot B^{\frac{3}{p}-1}_{p, 1}}^{\frac12}\|w\|_{\dot B^{\frac{3}{p}-1}_{p, 1}}^{\frac12}\|v\|_{\dot B^{\frac{3}{p}+1}_{p, 1}}^{\frac12}\|w\|_{\dot B^{\frac{3}{p}+1}_{p, 1}}^{\frac12},\label{3.6}
\end{align}
and, since ${\rm{div}}({\rm{curl}^{-1}}v)=0$, owing to Proposition \ref{P_25} {\it (viii)},
\begin{align}
\|{\rm{div}}(w\otimes ({\rm{curl}}^{-1}v))\|_{\dot B^{\frac{3}{p}}_{p, 1}}&=\|({\rm{curl}}^{-1}v)\cdot\nabla w\|_{\dot B^{\frac{3}{p}}_{p, 1}}\notag\\
&\lesssim\|{\rm{curl}}^{-1}v\|_{\dot B^{\frac{3}{p}}_{p, 1}}\|\nabla w\|_{\dot B^{\frac{3}{p}}_{p, 1}}\notag\\
&\lesssim\|v\|_{\dot B^{\frac{3}{p}-1}_{p, 1}}\|w\|_{\dot B^{\frac{3}{p}+1}_{p, 1}}.\label{3.7}
\end{align}

Hence, integrating on $\R_+$ and observing that the Leray projector $\mathcal P$ 
maps $\dot B^{\frac3p}_{p,1}$ to itself  according to Proposition \ref{P_25} (vii),  we get
\begin{align}
\|Q_{a}(v, w)\|_{{L}^{1}(\dot B^{\frac{3}{p}-1}_{p, 1})}&\lesssim\|{\rm{div}}(v\otimes w)+{\rm{div}}(w\otimes v)\|_{{L}^{1}(\dot B^{\frac{3}{p}-1}_{p, 1})}\notag\\
&\lesssim\|v\|_{X}\|w\|_{X},\label{3.8}
\end{align}
\begin{align}
\|Q_{b}(v, w)\|_{{L}^{1}(\dot B^{\frac{3}{p}-1}_{p, 1})}&=\|{\rm{div}}(v\otimes w)-{\rm{div}}(w\otimes v)\|_{{L}^{1}(\dot B^{\frac{3}{p}-1}_{p, 1})}\notag\\
&\lesssim\|v\|_{X}\|w\|_{X},\label{3.9}
\end{align}
\begin{align}
\!\!\|\nabla\!\times\! Q_{b}({\rm{Curl}}^{-1}&v, w)\|_{{L}^{1}(\dot B^{\frac{3}{p}-1}_{p, 1})}
\lesssim\|Q_{b}({\rm{Curl}}^{-1}v, w)\|_{{L}^{1}(\dot B^{\frac{3}{p}}_{p, 1})}\notag\\
\lesssim&\, \|{\rm{div}}(({\rm{Curl}}^{-1}v)\otimes w)\|_{{L}^{1}(\dot B^{\frac{3}{p}}_{p, 1})}+\|{\rm{div}}(w\otimes ({\rm{Curl}}^{-1}v))\|_{{L}^{1}(\dot B^{\frac{3}{p}}_{p, 1})}\quad\notag\\
\lesssim&\,\|v\|_{X}\|w\|_{X}.\label{3.10}
\end{align}
Now,  by  definition of $\mathcal B(V, W)$, we have
$$ \quad\left\{
\begin{aligned}
 &\partial_{t}\mathcal B(V, W)-\Delta \mathcal B(V, W) = Q(V, W), \\
&\mathcal B(V, W)|_{t=0} =0.
\end{aligned}
\right.$$
Hence, by Proposition \ref{Le_27} and the definition of $Q$ in  \eqref{eq:Q}, we get
$$\displaylines{\|\mathcal B(V, W)\|_{X}
\lesssim\| Q_{a}(V_{2}, W_{2})-Q_{a}(V_{1}, W_{1})\|_{{L}^{1}(\dot B^{\frac{3}{p}-1}_{p, 1})}
\hfill\cr\hfill+\|Q_{b}(V_{2}, W_{3}-W_{1})\|_{{L}^{1}(\dot B^{\frac{3}{p}-1}_{p, 1})}
+\|\nabla\times Q_{b}({\rm{curl}}^{-1}V_{3}, W_{3}-W_{1})\|_{{L}^{1}(\dot B^{\frac{3}{p}-1}_{p, 1})}.}$$
Remembering \eqref{3.8}-\eqref{3.10}, one  can conclude that $\mathcal B$
maps $X\times X$ to $X.$
Hence, System \eqref{eq:U} has a global solution $(u,B,J)$ in $X.$
\medbreak
For completing the proof of the global existence for the original Hall-MHD system, we  have to show that if $J_0=\nabla\times B_0,$ then   $J=\nabla\times B$
so that $(u,B)$ is indeed a distributional solution of \eqref{1.1}--\eqref{1.4}. 
Now, we have
$$(\partial_t-\Delta)(\nabla\times B-J)=\nabla\times Q_b({\rm{curl}}^{-1}(\nabla\times B-J),J-u).$$
Remember that $J-u$ belongs to $L^2(\dot B^{\frac 3p}_{p,1})$ (use interpolation for the space regularity
and H\"older inequality for the time variable),  that $J$ and $u$ are  in  $L^2_T(\dot B^{\frac 3p-1}_{p,1})$
for all $T>0$ since they are in $L^\infty(\dot B^{\frac 3p-1}_{p,1})$ and  observe 
that $\nabla\times B$ is in $L^2(\dot B^{\frac 3p-1}_{p,1}).$
Therefore, from the definition of $Q_b,$ the properties of continuity of operator $\curl^{-1},$ 
and product laws, we gather that 
$\nabla\times Q_b({\rm{curl}}^{-1}(\nabla\times B-J),J-u)$ is in $L^1_T(\dot B^{\frac 3p-2}_{p,1})$ for 
all $T>0.$ Because $(\nabla\times B-J)|_{t=0}=0,$  Proposition \ref{Le_27} thus guarantees
that $\nabla\times B-J$ is in $\cC([0,T];\dot B^{\frac 3p-2}_{p,1})$ for all $T>0.$
Furthermore, we have 
$$\displaylines{\|(\nabla\times B-J)(T)\|_{\dot B^{\frac3p-2}_{p,1}}
+\int_0^T\|\nabla\times B-J\|_{\dot B^{\frac3p}_{p,1}}\,dt\hfill\cr\hfill
\leq C\int_0^T\|J-u\|_{\dot B^{\frac 3p}_{p,1}}\|\nabla\times B- J\|_{\dot B^{\frac 3p-1}_{p,1}}\,dt.}$$
The right-hand side may be handled by means of an interpolation inequality: we get for all $\varepsilon>0,$
$$\|J-u\|_{\dot B^{\frac 3p}_{p,1}}\|\nabla\times B- J\|_{\dot B^{\frac 3p-1}_{p,1}}
\leq \eps\|\nabla\times B- J\|_{\dot B^{\frac 3p}_{p,1}}
+C\eps^{-1}\|J-u\|_{\dot B^{\frac 3p}_{p,1}}^2\|\nabla\times B- J\|_{\dot B^{\frac 3p-2}_{p,1}}.$$
Hence, taking $\eps$ small enough, then using   Gronwall lemma ensures that 
$\|(\nabla\times B- J)(t)\|_{\dot B^{\frac 3p-2}_{p,1}}=0$, whence
$\nabla\times B- J\equiv0$ a.e. on $\R_+\times\R^3.$ 
 This  yields the existence part of Theorem \ref{Th_1} in the small data case.

\medbreak
Let us explain how the above arguments have to be modified  so as to prove local
existence  in the case where only $J_0$ is small. 
The idea is to control the existence time  according to the solution
 $U^{L}$ of the heat equation:
$$ \quad\left\{
\begin{aligned}
 &\partial_{t}U^{L}-\Delta U^{L} = 0, \\
&U^{L}|_{t=0} = U_{0} .
\end{aligned}
\right.$$
By Proposition \ref{Le_27}, we have
\begin{equation}\label{3.12}
\|J^L\|_{{L}^{\infty}_{T}(\dot B^{\frac{3}{p}-1}_{p, 1})}\leq C\|J_{0}\|_{\dot B^{\frac{3}{p}-1}_{p, 1}},
\end{equation}
and, using also the dominated convergence theorem yields 
$$\lim\limits_{T\to 0}\|U^L\|_{{{L}^{\rho}_{T}(\dot B^{\frac{3}{p}+\frac{2}{\rho}-1}_{p, 1}})}=0,~~~\rm{whenever}~~1\leq\rho<\infty.$$
Clearly,  $U$ is a solution of \eqref{eq:U} on $[0, T]\times\mathbb R^3$ with data $U_0$ if and only if 
\begin{equation}\label{3.000}
U := U^L + \widetilde  U
\end{equation}
 with, for all  $t\in[0,T],$  
$$ 
 \widetilde U(t): = \int_0^t e^{(t-\tau)\Delta}(Q(\widetilde U, U^L) + Q(U^L, \widetilde U) + Q(\widetilde U, \widetilde U) + Q(U^L, U^L)) ~d\tau. 
$$
Then, proving local existence relies on  the following  generalization of  Lemma \ref{Le_31}.
\begin{lemma}\label{Le_220}
Let $(X, \|\cdot\|_{X})$ be a Banach space, $\mathcal B : X\times X\to X,$  a bilinear continuous operator with norm $K$ and $\mathcal L : X\to X$, a continuous linear operator with norm $M<1$. Let $y\in X$ satisfy  $4K\|y\|_{X}<(1-M)^2$. Then, equation $$x=y+{\mathcal L}(x)+\mathcal B(x,x)$$ has a unique solution $x$ in the ball $\mathit B(0, \frac{1-M}{2K})\cdotp$
\end{lemma}
Take $\cB$ as in \eqref{3.1}, set $y:= \mathcal B(U^L, U^L)$ and define the linear map $\mathcal L$ by 
\begin{equation}\label{eq:L}
\mathcal L(V) := 
 \mathcal B(V, U^L)+\mathcal B(U^L, V).\end{equation}
Our problem recasts in
\begin{equation}\label{3.555}
\widetilde U = y + \mathcal L(\widetilde U) + \mathcal B(\widetilde U, \widetilde U).
\end{equation}
For $X,$ we now take the space (denoted by $X_T$) of  triplets of 
divergence free vector-fields with components in $E_p(T).$
Then, arguing as for getting \eqref{3.5}, \eqref{3.6}, integrating on $[0,T]$ and using
Cauchy-Schwarz inequality, we get
$$\|\div(v\otimes w)\|_{{L}_T^{1}(\dot B^{\frac{3}{p}-1}_{p, 1})} 
+ \|\div(\curl^{-1}v\otimes w)\|_{{L}_T^{1}(\dot B^{\frac{3}{p}}_{p, 1})} 
\lesssim \|v\|_{L_T^2(\dot B^{\frac 3p}_{p,1})} \|w\|_{L_T^2(\dot B^{\frac 3p}_{p,1})}.
$$
Hence, using also \eqref{3.7} and the definition of $\cB(V,W),$ we end up with 
\begin{multline}\label{eq:BXT}
\|\mathcal B(V, W)\|_{X_T}
\lesssim\|V\|_{{L}^{2}_T(\dot B^{\frac{3}{p}}_{p, 1})}\|W\|_{{L}^{2}_T(\dot B^{\frac{3}{p}}_{p, 1})}\\+(\|W_1\|_{{L}^{1}_T(\dot B^{\frac{3}{p}+1}_{p, 1})}+\|W_3\|_{{L}^{1}_T(\dot B^{\frac{3}{p}+1}_{p, 1})})\|V_3\|_{{L}^{\infty}_T(\dot B^{\frac{3}{p}-1}_{p, 1})}.
\end{multline}
For justifying that $\cL$ defined in \eqref{eq:L} is indeed a continuous linear operator 
on $X_T$ \emph{with small norm if $T\to0,$} the troublemakers in the right-hand side of \eqref{eq:BXT} are 
\begin{equation*}
\|\widetilde u\|_{{L}^{1}_T(\dot B^{\frac{3}{p}+1}_{p, 1})}\|J^L\|_{{L}^{\infty}_T(\dot B^{\frac{3}{p}-1}_{p, 1})}\quad{\rm{and}}\quad\|\widetilde J\|_{{L}^{1}_T(\dot B^{\frac{3}{p}+1}_{p, 1})}\|J^L\|_{{L}^{\infty}_T(\dot B^{\frac{3}{p}-1}_{p, 1})}
\end{equation*}
since, for large $J_0,$  the term $\|J^L\|_{{L}^{\infty}_T(\dot B^{\frac{3}{p}-1}_{p, 1})}$ need 
not to be small.  One thus have to assume 
that  $\|J_0\|_{\dot B^{\frac{3}{p}-1}_{p, 1}}$  is small so as to guarantee that
the norm of $\cL$  is smaller than $1$  for $T$ small enough. Then,  
  one can conclude thanks to Lemma \ref{Le_220},  to the local-in-time existence statement  of Theorem~\ref{Th_1}. 
\bigbreak
To prove the uniqueness part of Theorem \ref{Th_1}. 
Consider two solutions $(u^1,B^1)$ and $(u^2,B^2)$ of \eqref{1.1}--\eqref{1.3} emanating 
from  the same data, and denote by  $U^{1}$ and $U^2$ the corresponding solutions 
of the extended system \eqref{eq:U}.  Since one can take (with no loss of generality)  for  $U^{2}$ the solution built previously,  and as  $\|J_0\|_{\dot B^{\frac{3}{p}-1}_{p, 1}}\leq c$  is assumed, we have  
\begin{equation}\label{eq:J2}
\|J^2\|_{X_T}\leq 2c.
\end{equation}
 Denoting $\dU := U^{2} - U^{1}$, we find that $\dU$ satisfies
$$ \partial_t {\dU}-\Delta {\dU}=Q(U^2, \dU)+Q(\dU, U^{1})$$
 with $\dU|_{t=0}{=}0,$
and thus 
\begin{equation*}
\dU= \cB(U^2, \dU)+\cB(\dU, U^{1}).
\end{equation*}
Arguing as in the proof of  \eqref{eq:BXT}  yields
$$\begin{aligned}
\| \cB(U^2, \dU)\|_{X_T}
&\lesssim\int_0^T \|U^2\|_{\dot B^{\frac3p}_{p,1}}\|\dU\|_{\dot B^{\frac3p}_{p,1}}\,dt
+\int_0^T \|J^2\|_{\dot B^{\frac3p-1}_{p,1}}\|\dU\|_{\dot B^{\frac3p+1}_{p,1}}\,dt\\
&\lesssim\int_0^T \|U^2\|_{\dot B^{\frac3p}_{p,1}}\|\dU\|_{\dot B^{\frac3p-1}_{p,1}}^{\frac12}
\|\dU\|_{\dot B^{\frac3p+1}_{p,1}}^{\frac12}\,dt
+\int_0^T \|J^2\|_{\dot B^{\frac3p-1}_{p,1}}\|\dU\|_{\dot B^{\frac3p+1}_{p,1}}\,dt
\end{aligned}
$$
whence there exists $C>0$ such that  for all $\eta>0,$ 
$$\displaylines{
\| \cB(U^2, \dU)\|_{X_T} \leq \bigl(\eta + C\|J^2\|_{L_T^\infty(\dot B^{\frac3p-1}_{p,1})}\bigr)
\|\dU\|_{L^1_T(\dot B^{\frac{3}{p}+1}_{p, 1})}
\hfill\cr\hfill+C\eta^{-1}\int_0^T \|U^2\|_{\dot B^{\frac3p}_{p,1}}^2\|\dU\|_{\dot B^{\frac3p-1}_{p,1}}\,dt.}
$$
Similarly, we have
$$\begin{aligned}
\| \cB(\dU,U^1)\|_{X_T}&\leq
C\biggl(\int_0^T \|U^1\|_{\dot B^{\frac3p}_{p,1}}\|\dU\|_{\dot B^{\frac3p}_{p,1}}\,dt
+\int_0^T  \|U^1\|_{\dot B^{\frac3p+1}_{p,1}}\|\dJ\|_{\dot B^{\frac3p-1}_{p,1}}\,dt
\biggr)\\
&\leq \eta
\|\dU\|_{L^1_T(\dot B^{\frac{3}{p}+1}_{p, 1})}+
C\int_0^T\bigl( \|U^1\|_{\dot B^{\frac3p+1}_{p,1}}+ \eta^{-1}\|U^1\|_{\dot B^{\frac3p}_{p,1}}^2\bigr)
\|\dU\|_{\dot B^{\frac3p-1}_{p,1}}\,dt.\end{aligned}
$$
Hence, taking $\eta$ small enough, and remembering \eqref{eq:J2},  one gets
 $$
\|\dU\|_{X_T}\leq C\int_0^T \bigl( \|U^1\|_{\dot B^{\frac3p+1}_{p,1}}+ \|U^1\|_{\dot B^{\frac3p}_{p,1}}^2+  \|U^2\|_{\dot B^{\frac3p}_{p,1}}^2\bigr)
\|\dU\|_{\dot B^{\frac3p-1}_{p,1}}\,dt.
$$
 Gronwall lemma thus implies that $\dU\equiv 0$ in $X_T,$ 
whence uniqueness on $[0,T]\times\R^3.$
Of course, in the case where the data are small, then $J^2$ remains 
small for all $T>0,$ and one  gets uniqueness on $\R_+\times\R^3.$ 

%
\bigbreak
Let us finally justify  the propagation of Sobolev regularity
in the case where,  additionally, $(u_0,B_0)$ is in $H^s\times H^r$ with $(r,s)$ satisfying \eqref{cond:reg}. 
For expository purpose,  assume that the data fulfill \eqref{eq:smallness}
(the case where only $J_0$ is small being left to the reader).
Our aim is to prove that the solution $(u,B)$ we constructed above satisfies
$$ (u,B)\in \cC_b(\R_+;H^s\times H^r)\andf (\nabla u,\nabla B)\in L^2(\R_+;H^s\times H^r).$$ 
For the time being, let us assume that $(u,B)$ is smooth.
Then,  taking the $L^2$ scalar product of  \eqref{1.1} and \eqref{1.3} by $u$ and $B$, respectively, adding up the resulting identities, and using the fact that  
\begin{equation*}
(\nabla\times(J\times B),~B)=(J\times B,~J)=0,
\end{equation*}
one gets  the following energy balance:
\begin{equation}
\frac{1}{2}\frac{d}{dt}(\|u\|_{L^2}^2+\|B\|_{L^2}^2)+\|\nabla u\|_{L^2}^2+\|\nabla B\|_{L^2}^2=0\label{3.1500}.
\end{equation}
Since $\|z\|_{\dot H^{a}}= \|\Lambda^a z\|_{L^2}$ and 
$\|z\|_{H^a}\sim \|z\|_{L^2} +  \|z\|_{\dot H^{a}},$
in order to prove  estimates in $H^s\times H^r,$ it suffices
to get a suitable control on $\|\Lambda^s u\|_{L^2}$ and on $\|\Lambda^rB\|_{L^2}.$
To this end, apply  $\Lambda^s$ to  \eqref{1.1}, then take the $L^2$ scalar product with $\Lambda^s u$: 
\begin{align}
\frac{1}{2}\frac{d}{dt}\|\Lambda^s u\|_{L^2}^2+\|\Lambda^s\nabla u\|_{L^2}^2&=(\Lambda^s(B\cdot\nabla B), \Lambda^s u)-(\Lambda^s(u\cdot\nabla u), \Lambda^s u)\notag\\
&=:A_1+A_2.\notag
\end{align}
Similarly, apply   $\Lambda^r$ to \eqref{1.3} and take the  $L^2$ scalar product with $\Lambda^r B$
to get:
\begin{align}
\frac{1}{2}\frac{d}{dt}\|\Lambda^r B\|_{L^2}^2+\|\Lambda^r\nabla B\|_{L^2}^2&=(\Lambda^r(u\times B), \Lambda^r J)-(\Lambda^r(J\times B), \Lambda^rJ)\notag\\
&=:A_3+A_4.\notag
\end{align}
To bound $A_1,$ $A_2,$ $A_3$ and $A_4,$ we shall use repeatedly the following 
classical tame estimate in homogeneous Sobolev spaces:
\begin{equation}\label{eq:tame}
\|\Lambda^\sigma(fg)\|_{L^2}\lesssim \|f\|_{L^\infty}\|\Lambda^\sigma g\|_{L^2}
+  \|g\|_{L^\infty}\|\Lambda^\sigma f\|_{L^2},\qquad\sigma\geq0.
\end{equation}
Using  first the Cauchy-Schwarz inequality, then 
\eqref{eq:tame}, the fact that  $s\leq r\leq 1+s$  and Young inequality, we readily get
\begin{align}
|A_1|&\leq C(\|\Lambda^s B\|_{L^2}\|\nabla B\|_{L^\infty}+\|B\|_{L^\infty}\|\Lambda^s\nabla B\|_{L^2})\|u\|_{H^s}\notag\\
&\leq C(\|B\|_{H^s}^2+\|u\|_{H^s}^2)\|\nabla B\|_{L^\infty}+\frac{1}{8}\|\nabla B\|_{H^r}^2+C\|B\|_{L^\infty}^2\|u\|_{H^s}^2,\notag\\
|A_2|&\leq C\|\nabla u\|_{L^\infty}\|u\|_{H^s}^2,\notag\\
|A_3|&\leq C\|\Lambda^r (u\times B)\|_{L^2}\|\Lambda^r J\|_{L^2}\notag\\
&\leq C(\|\Lambda^r u\|_{L^2}^2\|B\|_{L^\infty}^2+\|\Lambda^r B\|_{L^2}^2\|u\|_{L^\infty}^2)+\frac{1}{8}\|\nabla B\|_{H^r}^2\notag\\
&\leq C\bigl(\|u\|_{L^2}^2+\|\nabla u\|_{H^s}^2\bigr)\|B\|_{L^\infty}^2+C\|B\|_{H^r}^2\|u\|_{L^\infty}^2+\frac{1}{8}\|\nabla B\|_{H^r}^2,\notag\\
|A_4|&\leq C\|J\times B\|_{H^r}\|J\|_{H^r}\notag\\
&\leq C(\|J\|_{H^r}^2\|B\|_{L^\infty}+\|J\|_{L^\infty}\|B\|_{H^r}\|J\|_{H^r})\notag\\
&\leq C\|B\|_{L^\infty}\|\nabla B\|_{H^r}^2
+C\|J\|_{L^\infty}^2\|B\|_{H^r}^2+\frac{1}{8}\|J\|_{H^r}^2.\notag
\end{align}
Putting the above estimates  and \eqref{3.1500} together, and using the fact that 
$\|B\|_{L^\infty}$ is small since, according to Proposition \ref{P_25} and 
the first part of the proof, we have  
$$\|B\|_{L^\infty}\lesssim \|B\|_{\dot B^{\frac 3p}_{p,1}}\lesssim \|J\|_{\dot B^{\frac 3p-1}_{p,1}}
\lesssim \|(u_0,B_0,J_0)\|_{\dot B^{\frac 3p-1}_{p,1}},$$ one gets
$$\frac{1}{2}\frac{d}{dt}(\|u\|_{H^s}^2+\|B\|_{H^r}^2)+\|\nabla u\|_{H^s}^2+\|\nabla B\|_{H^r}^2\notag\\
\leq C (\|u\|_{H^s}^2+\|B\|_{H^r}^2)S(t),$$
with
\begin{align*}
S(t):=\|\nabla u(t)\|_{L^\infty}+\|\nabla B(t)\|_{L^\infty}+\|u(t)\|_{L^\infty}^2+\|B(t)\|_{L^\infty}^2+\|J(t)\|_{L^\infty}^2.
\end{align*}
By Gronwall's inequality, we conclude that  for all $t\geq 0,$
$$\displaylines{\quad
\|u(t)\|_{H^s}^2+\|B(t)\|_{H^r}^2+\int^t_0(\|\nabla u(\tau)\|_{H^s}^2+\|\nabla B(\tau)\|_{H^r}^2)\,d\tau\hfill\cr\hfill
\leq \bigl(\|u_0\|_{H^s}^2+\|B_0\|_{H^r}^2\bigr)\exp\biggl(C\int_0^t S(\tau)\,d\tau\biggr)\cdotp\quad}
$$
As  $\int_0^t S(\tau)\,d\tau$ is bounded thanks to the first part of the theorem and 
embedding (use Proposition \ref{P_25} (ii)),
we get a control of the Sobolev norms for all time. 
\medbreak
Let us briefly explain how those latter computations
may be made rigorous. 
Let us consider data $(u_0,B_0)$ fulfilling \eqref{eq:smallness} and such that, 
additionally, we have $u_0$ in $H^s$ and $B_0$ in $H^r$ with $(r,s)$ satisfying \eqref{cond:reg}. 
Then, there exists a sequence $(u_0^n,B_0^n)$ in the Schwartz space $\cS$  such that 
$$
(u_0^n,B_0^n)\to (u_0,B_0)\ \hbox{ in }\ \bigl(\dot B^{\frac3p-1}_{p,1}\cap H^s\bigr)\times\bigl(\dot B^{\frac3p-1}_{p,1}\cap B^{\frac3p}_{p,1}\cap H^r\bigr)\cdotp$$
The classical well-posedness theory in Sobolev spaces (see e.g. \cite{Ch13})  ensures that 
the Hall-MHD system with data $(u_0^n,B_0^n)$  has a unique maximal solution $(u^n,B^n)$ 
on some interval $[0,T^n)$ belonging to all Sobolev spaces.
For that solution, the previous computations hold, and one
ends up for all $t<T^n$ with 
$$\displaylines{\quad
\|u^n(t)\|_{H^s}^2+\|B^n(t)\|_{H^r}^2+\int^t_0(\|\nabla u^n(\tau)\|_{H^s}^2+\|\nabla B^n(\tau)\|_{H^r}^2)\,d\tau\hfill\cr\hfill
\leq \bigl(\|u^n_0\|_{H^s}^2+\|B^n_0\|_{H^r}^2\bigr)\exp\biggl(C\int_0^t S^n(\tau)\,d\tau\biggr),\quad}
$$
where 
$$
S^n(t):=\|\nabla u^n(t)\|_{L^\infty}+\|\nabla B^n(t)\|_{L^\infty}
+\|u^n(t)\|_{L^\infty}^2+\|B^n(t)\|_{L^\infty}^2+\|J^n(t)\|_{L^\infty}^2.
$$
Since the regularized data  $(u_0^n,B_0^n)$ fulfill \eqref{eq:smallness}  for large enough $n,$
they generate a global solution $(\wt u^n,\wt B^n)$ in $E_p$ which, actually, coincides with $(u^n,B^n)$
on $[0,T^n)$
by virtue  of the uniqueness result  that has been proved before. 
Therefore, $S^n$ belongs to $L^1(0,T^n)$ and thus
$(u^n,B^n)$ is in $L^\infty(0,T^n;H^s\times H^r).$ 
Combining with the  continuation argument of e.g. \cite{Ch13}, one can conclude that $T^n=+\infty.$ 
\smallbreak
At this stage, one can assert that:
\begin{enumerate}
\item[i)] $(u^n,B^n,J^n)_{n\in\N}$ is bounded in $E_p$;
\item[ii)]  $(u^n,B^n)_{n\in\N}$ is bounded in $\cC(\R_+;H^s\times H^r)$
and  $(\nabla u^n,\nabla B^n)_{n\in\N}$ is bounded in $L^2(\R_+;H^s\times H^r).$
\end{enumerate}
Hence, up to subsequence, 
\begin{enumerate}
\item[i)] $(u^n,B^n,J^n)$ converges weakly $*$ in  $E_p$;
\item[ii)] $(u^n,B^n)$ converges weakly $*$ in $L^\infty(\R_+;H^s\times H^r);$
\item[iii)] $(\nabla u^n,\nabla B^n)$ converges weakly  in $L^2(\R_+;H^s\times H^r).$
 \end{enumerate}
 Clearly, a small variation of the proof of uniqueness in $E_p$ allows
 to prove the continuity of the flow map. 
 Hence, given that $(u^n_0,B^n_0,J^n_0)$ converges to $(u_0,B_0,J_0)$
 in $\dot B^{\frac3p-1}_{p,1},$ one gets  $(u^n,B^n,J^n)\to (u,B,J)$ strongly in $E_p,$ 
 where  $(u,B,J)$ stands for the solution of \eqref{eq:U} with data $(u_0,B_0,J_0).$ 
 
 Since the weak convergence  results listed above imply  the convergence in the sense of distributions, one can conclude  that the weak limit coincides with the strong one in $E_p.$ 
 Hence $(u,B)$ (resp. $(\nabla u,\nabla B)$)
 is indeed in   $L^\infty(\R_+;H^s\times H^r)$
 (resp. $L^2(\R_+;H^s\times H^r)$). 
  Then,  looking at $(u,B)$ as the solution of a heat equation with right-hand side 
  in $L^2(\R_+; H^{s-1}\times H^{r-1})$ yields 
 the time continuity  with values in Sobolev spaces (use for instance Proposition \ref{Le_27}). 
 This completes the proof of Theorem \ref{Th_1}.\quad$\square$


\section{Local existence for large data in \texorpdfstring{$\dot B^{\frac12}_{2,1},$}{TEXT}
and blow-up criteria}\label{s:new}

Proving  Theorem \ref{Th_2} is based on a priori estimates in the space 
$E_2(T)$ for  smooth solutions $(u,B,v)$ of  \eqref{1.100}.
Those estimates will be obtained by implementing an energy method on  \eqref{1.100}
after localization in the Fourier space.
A slight modification of the method will yield uniqueness and blow-up criteria. 
 
 Throughout this  section and the following one, we shall 
 take advantage of the rescaling \eqref{eq:rescaling},  so as to  reduce
 our study to the case  $\mu=\nu=h=1$ 
 (remember that we have $\mu/\nu=1$  in Theorem \ref{Th_2}).

\medbreak\noindent\textbf{First step: A priori estimates.}

 Our main aim here is to prove the following result. 
 \begin{proposition}\label{P_41}
 Consider a smooth solution $(u,B,P)$ to the Hall-MHD System on $[0,T]\times\R^3$ for some $T>0,$ and denote $v:=u-\nabla\times B.$
 Let  $u^L:=e^{t\Delta} u_0,$ $B^L:=e^{t\Delta} B_0,$ 
$v^L:=e^{t\Delta} v_0$ and  $(\wt u, \wt B, \wt v) := (u-u^{L}, B-B^{L}, v-v^{L})$.
Let  
\begin{align*}
&c_1(t) :=\|v^L(t)\|_{\dot B^{\frac{5}{2}}_{2, 1}},\\
&c_2(t) :=\|u^L(t)\|_{\dot B^{\frac{3}{2}}_{2, 1}}^2+\|B^L(t)\|_{\dot B^{\frac{3}{2}}_{2, 1}}^2+\bigl(\|u_0\|_{\dot B^{\frac{1}{2}}_{2, 1}}+\|v_0\|_{\dot B^{\frac{1}{2}}_{2, 1}}\bigr)\|v^L(t)\|_{\dot B^{\frac{5}{2}}_{2, 1}}.
\end{align*}
 There exist three positive constants $\kappa,$ $C$ and $C_1$  such that if
 \begin{equation}\label{eq:X2}
\int_0^Tc_2(\tau)e^{C\int_\tau^T c_1(\tau')\,d\tau'}\,d\tau <\kappa,
\end{equation}
then we have 
\begin{eqnarray}
\|(\widetilde u, \widetilde B, \widetilde v)\|_{L^{\infty}_T(\dot B^{\frac{1}{2}}_{2, 1})}+C_1\|(\widetilde u, \widetilde B, \widetilde v)\|_{L^{1}_T(\dot B^{\frac{5}{2}}_{2, 1})}&\!\!\!\leq\!\!\!& C\kappa\label{4.999}
\andf\\
\|(u, B, v)\|_{L^{\infty}_T(\dot B^{\frac{1}{2}}_{2, 1})}+C_1\|(u, B, v)\|_{L^{1}_T(\dot B^{\frac{5}{2}}_{2, 1})}
&\!\!\!\leq\!\!\!& \|(u_0,B_0,v_0)\|_{\dot B^{\frac{1}{2}}_{2, 1}}+C\kappa.\qquad\label{4.001}
\end{eqnarray}
\end{proposition}
\begin{proof}
{}From \eqref{eq:heat}, Plancherel identity and the definition of $\|\cdot\|_{\dot B^s_{2,1}},$
we have  for some universal constant $C_1,$ 
\begin{equation}\label{eq:free} 
\|z\|_{L^{\infty}_T(\dot B^{\frac{1}{2}}_{2, 1})}+C_1\|z\|_{L^{1}_T(\dot B^{\frac{5}{2}}_{2, 1})}
\leq \|z_0\|_{\dot B^{\frac{1}{2}}_{2, 1}}\quad\hbox{for }\ 
z=u^L,B^L,v^L.
\end{equation}
Hence Inequality \eqref{4.001} follows from Inequality \eqref{4.999}. 
\medbreak
In order to prove  \eqref{4.999}, we use the fact that 
 $(\widetilde u, \widetilde B, \widetilde v,Q)$  satisfies 
\begin{equation}\label{4.1}
\left\{
\begin{aligned}
&\partial_t \widetilde u-\Delta\widetilde u=B\cdot\nabla B-u\cdot\nabla u-\nabla Q,\\
&\partial_t{\widetilde B}-\Delta\widetilde B=\nabla\times(v\times B),\\
&\partial_t\widetilde v-\Delta\widetilde v
= B\cdot\nabla B-u\cdot\nabla u
-\nabla\times((\nabla\times\widetilde v)\times B)\\&\qquad-\nabla\times((\nabla\times v^L)\times B)+\nabla\times(v\times u)+2\nabla\times(v\cdot\nabla B)-\nabla Q,
\end{aligned}
\right.
\end{equation}
with null initial condition.\medbreak

Apply  operator $\dot\Delta_j$ to both sides of \eqref{4.1}, then take the $L^2$ scalar product with $\dot\Delta_j \widetilde u$, $\dot\Delta_j\widetilde B$, $\dot\Delta_j\widetilde v,$ respectively.  To handle the third equation of \eqref{4.1}, use that 
$$\nabla\times\dot\Delta_j((\nabla\times\widetilde v)\times B)\\
=\nabla\times([\dot\Delta_j, B\times](\nabla\times\widetilde v))+\nabla\times(B\times\dot\Delta_j(\nabla\times\widetilde v)),
$$
and that the $L^2$ scalar product of the last term with $\dot\Delta_j\wt v$ is 0. 
Then, we get
$$\begin{aligned}
\frac{1}{2}\frac{d}{dt}\|\dot\Delta_j\widetilde u\|_{L^2}^2 +\|\nabla\dot\Delta_j\widetilde u\|_{L^2}^2
&\leq(\|\dot\Delta_j (B\cdot\nabla B)\|_{L^2}+\|\dot\Delta_j (u\cdot\nabla u)\|_{L^2})\|\dot\Delta_j\widetilde u\|_{L^2},\\
\frac{1}{2}\frac{d}{dt}\|\dot\Delta_j\widetilde B\|_{L^2}^2 +\|\nabla\dot\Delta_j\widetilde B\|_{L^2}^2
&\leq \|\nabla\times\dot\Delta_j (v\times B)\|_{L^2}\|\dot\Delta_j\widetilde B\|_{L^2},\\
\frac{1}{2}\frac{d}{dt}\|\dot\Delta_j\widetilde v\|_{L^2}^2 +\|\nabla\dot\Delta_j\widetilde v\|_{L^2}^2&\leq\bigl(\|\dot\Delta_j (B\cdot\nabla B)\|_{L^2}
+\|\dot\Delta_j (u\cdot\nabla u)\|_{L^2}\bigr)\|\dot\Delta_j\widetilde v\|_{L^2}\\
+\bigl(\|[\dot\Delta_j, B\times](&\nabla\times\widetilde v)\|_{L^2}
+\|\dot\Delta_j((\nabla\times v^L)\times B)\|_{L^2}+\|\dot\Delta_j(v\times u)\|_{L^2}\\
&\hspace{3cm}+2\|\dot\Delta_j(v\cdot\nabla B)\|_{L^2}\bigr)\|\nabla\times\dot\Delta_j\widetilde v\|_{L^2}.\end{aligned}
$$
Hence, using   Bernstein inequalities, one can deduce after time integration  that
for some universal constants $C_1$ and $C_2,$ 
\begin{multline}
\|(\dot\Delta_j\widetilde u,\dot\Delta_j\widetilde B,\dot\Delta_j\widetilde v)(t)\|_{L^2}
+C_12^{2j}\int_{0}^t \|(\dot\Delta_j\widetilde u,\dot\Delta_j\widetilde B,\dot\Delta_j\widetilde v)\|_{L^2}\,d\tau\\
\leq \int_0^t\Bigl(\|\dot\Delta_j (B\cdot\nabla B)\|_{L^2}+\|\dot\Delta_j (u\cdot\nabla u)\|_{L^2}
+C_22^j\Bigl(\|[\dot\Delta_j, B\times](\nabla\times\widetilde v)\|_{L^2}\\
+\|\dot\Delta_j((\nabla\times v^L)\times B)\|_{L^2}+\|\dot\Delta_j(v\times u)\|_{L^2}\\
+\|\dot\Delta_j(v\cdot\nabla B)\|_{L^2}+\|\dot\Delta_j (v\times B)\|_{L^2})\,d\tau\Bigr)\Bigr)\,d\tau.\label{4.2}
\end{multline}
Multiplying  both sides of \eqref{4.2} by $2^\frac{j}{2}$ and summing up  over $j\in\mathbb{Z}$, we obtain that
\begin{multline}
\|(\widetilde u,\widetilde B,\widetilde v)(t)\|_{\dot B^{\frac{1}{2}}_{2, 1}}
+C_1\int_0^t\|(\widetilde u,\widetilde B,\widetilde v)\|_{\dot B^{\frac{5}{2}}_{2, 1}}\,d\tau\\
\leq C_2\int_0^t \Bigl(\|B\cdot\nabla B\|_{\dot B^{\frac{1}{2}}_{2, 1}}+\|u\cdot\nabla u\|_{\dot B^{\frac{1}{2}}_{2, 1}}+\|v\times B\|_{\dot B^{\frac{3}{2}}_{2, 1}}+\|v\times u\|_{\dot B^{\frac{3}{2}}_{2, 1}}
\\+\|v\cdot\nabla B\|_{\dot B^{\frac{3}{2}}_{2, 1}}+\|(\nabla\times v^L)\times B\|_{\dot B^{\frac{3}{2}}_{2, 1}}
+\sum_j 2^\frac{3j}{2}\|[\dot\Delta_j, B\times](\nabla\times\widetilde v)\|_{L^2}\Bigr)\,d\tau.\label{4.5}
\end{multline}
Using \eqref{prod1}, Proposition \ref{P_25} ${\rm{(i)}}$, ${\rm{(ii)}}$, ${\rm{(iii)}}$ and Young's inequality yields
$$\begin{aligned}
\|B\cdot\nabla B\|_{\dot B^{\frac{1}{2}}_{2, 1}}
&\lesssim\|B^L\|_{\dot B^{\frac{3}{2}}_{2, 1}}^2+\|\tilde B\|_{\dot B^{\frac{3}{2}}_{2, 1}}^2\\
&\lesssim\|B^L\|_{\dot B^{\frac{3}{2}}_{2, 1}}^2+\|\widetilde B\|_{\dot B^{\frac{1}{2}}_{2, 1}}\|\widetilde B\|_{\dot B^{\frac{5}{2}}_{2, 1}},
\end{aligned}$$
$$\begin{aligned}
\|u\cdot\nabla u\|_{\dot B^{\frac{1}{2}}_{2, 1}}\!+\!\|v\times B\|_{\dot B^{\frac{3}{2}}_{2, 1}}\!+\!\|v\times u\|_{\dot B^{\frac{3}{2}}_{2, 1}}
&\lesssim \|u\|_{\dot B^{\frac{3}{2}}_{2, 1}}^2
+\bigl(\|B\|_{\dot B^{\frac{3}{2}}_{2, 1}}+\|u\|_{\dot B^{\frac{3}{2}}_{2, 1}}\bigr)
\|v\|_{\dot B^{\frac{3}{2}}_{2, 1}}\\
&\lesssim \|u^L\|_{\dot B^{\frac{3}{2}}_{2, 1}}^2+\|B^L\|_{\dot B^{\frac{3}{2}}_{2, 1}}^2\!+\!\|v^L\|_{\dot B^{\frac{1}{2}}_{2, 1}}\|v^L\|_{\dot B^{\frac{5}{2}}_{2, 1}}\\
+\|\widetilde u\|_{\dot B^{\frac{1}{2}}_{2, 1}}\|\widetilde u\|_{\dot B^{\frac{5}{2}}_{2, 1}}&+\|\widetilde B\|_{\dot B^{\frac{1}{2}}_{2, 1}}\|\widetilde B\|_{\dot B^{\frac{5}{2}}_{2, 1}}+\|\widetilde v\|_{\dot B^{\frac{1}{2}}_{2, 1}}\|\widetilde v\|_{\dot B^{\frac{5}{2}}_{2, 1}}.
\end{aligned}$$
Using that $B={\rm{curl}}^{-1}(u-v)$ and that
  $\nabla{\rm{curl}}^{-1}$ is a self-map on $\dot B^{\frac{3}{2}}_{2, 1}$
(see Proposition \ref{P_25} ${\rm{(vii)}}$) yields
$$\begin{aligned}
\|v\cdot\nabla B\|_{\dot B^{\frac{3}{2}}_{2, 1}}
&\lesssim \|v\|_{\dot B^{\frac{3}{2}}_{2, 1}}\|\nabla {\rm{curl}}^{-1}(u-v)\|_{\dot B^{\frac{3}{2}}_{2, 1}}\\
&\lesssim \|v\|_{\dot B^{\frac{3}{2}}_{2, 1}}^2+\|u\|_{\dot B^{\frac{3}{2}}_{2, 1}}^2\\
&\lesssim\|u^L\|_{\dot B^{\frac{3}{2}}_{2, 1}}^2+\|v^L\|_{\dot B^{\frac{1}{2}}_{2, 1}}\|v^L\|_{\dot B^{\frac{5}{2}}_{2, 1}}+\|\widetilde u\|_{\dot B^{\frac{1}{2}}_{2, 1}}\|\widetilde u\|_{\dot B^{\frac{5}{2}}_{2, 1}}+\|\widetilde v\|_{\dot B^{\frac{1}{2}}_{2, 1}}\|\widetilde v\|_{\dot B^{\frac{5}{2}}_{2, 1}}
\end{aligned}$$
and, using also  \eqref{prod1}, 
$$\begin{aligned}
\|(\nabla\!\times\! v^L)\!\times\! B\|_{\dot B^{\frac{3}{2}}_{2, 1}}
&\!\lesssim \|\nabla\times v^L\|_{\dot B^{\frac{3}{2}}_{2, 1}}\|B\|_{\dot B^{\frac{3}{2}}_{2, 1}}\\
&\!\lesssim \|v^L\|_{\dot B^{\frac{5}{2}}_{2, 1}}\|{\rm{curl}}^{-1}(u-v)\|_{\dot B^{\frac{3}{2}}_{2, 1}}\\
&\!\lesssim \|v^L\|_{\dot B^{\frac{5}{2}}_{2, 1}}\bigl(\|u^L\|_{\dot B^{\frac{1}{2}}_{2,1}}
+\|v^L\|_{\dot B^{\frac{1}{2}}_{2, 1}}\!+\!\|\wt u\|_{\dot B^{\frac{1}{2}}_{2, 1}}
\!+\!\|\wt v\|_{\dot B^{\frac{1}{2}}_{2, 1}}\bigr)\cdotp
\end{aligned}$$
{}From  the estimate \eqref{com01} with $s=3/2$ and the embedding 
$\dot B^{\frac32}_{2,1}\hookrightarrow L^\infty,$ we get
\begin{equation}\label{com00}
\sum_j 2^\frac{3j}{2}\|[\dot\Delta_j, b]a\|_{L^2}
\lesssim \|\nabla b\|_{\dot B^{\frac{3}{2}}_{2, 1}}\|a\|_{\dot B^{\frac{1}{2}}_{2, 1}},
\end{equation}
whence 
$$\begin{aligned}
\sum_j 2^\frac{3j}{2}\|[\dot\Delta_j, B\times](\nabla\times \widetilde v)\|_{L^2}
&\lesssim\|v-u\|_{\dot B^{\frac{3}{2}}_{2, 1}}^2+\|\widetilde v\|_{\dot B^{\frac{3}{2}}_{2, 1}}^2\\
&\lesssim \|u^L\|_{\dot B^{\frac{3}{2}}_{2, 1}}^2+\|v^L\|_{\dot B^{\frac{1}{2}}_{2, 1}}\|v^L\|_{\dot B^{\frac{5}{2}}_{2, 1}}
\\&\qquad\qquad+\|\widetilde u\|_{\dot B^{\frac{1}{2}}_{2, 1}}\|\widetilde u\|_{\dot B^{\frac{5}{2}}_{2, 1}}+\|\widetilde v\|_{\dot B^{\frac{1}{2}}_{2, 1}}\|\widetilde v\|_{\dot B^{\frac{5}{2}}_{2, 1}}.
\end{aligned}$$
Plugging the above estimates into the right-hand side of \eqref{4.5} and using 
\eqref{eq:free},  we end up with
\begin{align}\label{4.6}
&X(t)+C_1\int_0^t D(\tau)\,d\tau\leq C\int_0^t X(\tau)D(\tau)\,d\tau +C\int_0^t(c_1(\tau)X(\tau)+c_2(\tau))\,d\tau,
\end{align}
where $c_1$ and $c_2$ have been defined in the proposition, 
\begin{align*}
&X(t) :=\|\widetilde u(t)\|_{\dot B^{\frac{1}{2}}_{2, 1}}+\|\widetilde B(t)\|_{\dot B^{\frac{1}{2}}_{2, 1}}+\|\widetilde v(t)\|_{\dot B^{\frac{1}{2}}_{2, 1}}\\
\andf&D(t) :=\|\widetilde u(t)\|_{\dot B^{\frac{5}{2}}_{2, 1}}+\|\widetilde B(t)\|_{\dot B^{\frac{5}{2}}_{2, 1}}+\|\widetilde v(t)\|_{\dot B^{\frac{5}{2}}_{2, 1}}.
\end{align*}
Note that whenever
\begin{equation}\label{eq:X1}
2C\sup_{\tau\in[0,t]} X(\tau)\leq C_1,
\end{equation}
Inequality \eqref{4.6} combined with Gronwall lemma  implies that
\begin{equation}
X(t)+\frac{C_1}2\int_0^t D(\tau)\,d\tau\leq C\int_0^tc_2(\tau)e^{C\int_\tau^t c_1(\tau')\,d\tau'}\,d\tau.
\end{equation}
Now, if  Condition \eqref{eq:X2} is satisfied with $\kappa:=C_1/2C^2,$ 
then the fact that the left-hand side of \eqref{4.6} is a continuous function on $[0,T]$
that vanishes at $0$ combined with a standard bootstrap argument 
allows to prove that \eqref{eq:X1} and thus \eqref{eq:X2} is satisfied. 
Renaming the constants completes the proof of the proposition. 
\end{proof}

\noindent\textbf{Second step: Constructing approximate solutions.}
It is based on Friedrichs' method : consider the spectral cut-off operator $\mathbb{E}_n$ defined by 
$$\mathcal F({\mathbb{E}_n}f)(\xi)=1_{\{n^{-1}\leq|\xi|\leq n\}}(\xi)\mathcal F(f)(\xi).$$
We  want to solve the following truncated system:
\begin{equation}\label{4.000}
\left\{
\begin{aligned}
&\partial_t u-\Delta u=\mathbb{E}_n\mathcal P(\mathbb{E}_nB\cdot\mathbb{E}_n\nabla B-\mathbb{E}_nu\cdot\nabla \mathbb{E}_nu),\\
&\partial_t{\mathnormal B}-\Delta B=\nabla\times\mathbb{E}_n(\mathbb{E}_n(u-\nabla\times B)\times\mathbb{E}_nB),\\
\end{aligned}\right.\end{equation}
supplemented with initial data $(\mathbb{E}_n\mathnormal u_{0}, \mathbb{E}_n\mathnormal B_{0}).$
\medbreak
We need  the  following obvious lemma:
\begin{lemma}
Let $s\in\mathbb{R}$ and  $k\geq 0$. Let $f\in \dot{B}_{2, 1}^s.$ Then, for all $n\geq1$, we have
\begin{equation}
\|\mathbb{E}_n f\|_{\dot{B}_{2, 1}^{s+k}}\lesssim n^k\|f\|_{\dot{B}_{2, 1}^s},\label{4.100}
\end{equation}
\begin{equation}
\lim\limits_{n\to \infty}\|\mathbb{E}_n f-f\|_{\dot{B}_{2, 1}^{s}}=0\label{4.200},
\end{equation}
\begin{equation}
\|\mathbb{E}_n f-f\|_{\dot{B}_{2, 1}^{s}}\lesssim \frac{1}{n^k}\|f\|_{\dot{B}_{2, 1}^{s+k}}\label{4.201}.
\end{equation}
\end{lemma}
We claim that \eqref{4.000} is an ODE  in the Banach space $L^2(\R^3;\R^3\times\R^3)$ 
for which the standard Cauchy-Lipschitz theorem applies.  
Indeed, the above lemma ensures that $\mathbb{E}_n$ maps $L^2$ to all Besov spaces, 
and   that the right-hand side of  \eqref{4.000} is a continuous bilinear map 
from   $L^2(\R^3;\R^3\times\R^3)$  to itself. 
We thus deduce that  \eqref{4.000} admits a unique maximal solution 
$(u^n,B^n)\in \cC^1([0,T^n); L^2(\R^3;\R^3\times\R^3)).$
Furthermore, as $\mathbb{E}_n^2=\mathbb{E}_n,$ uniqueness implies  $\mathbb{E}_n u^n=u^n$
and  $\mathbb{E}_n B^n=B^n,$ and we clearly have $\div u^n=\div B^n=0.$
Being spectrally supported in the annulus $\{n^{-1}\leq|\xi|\leq n\},$ 
one can also deduce that the solution belongs to $\cC^1([0,T^n); \dot B^s_{2,1})$ for all $s\in\R.$ 
Hence, setting $J^n:=\nabla\times B^n$ and $v^n:=u^n-J^n,$  we see that 
$u^n,$ $B^n$ and $v^n$ belong to the space $E_{2}(T)$ for all $T<T^n$ and fulfill:
\begin{equation}\label{4.000b}
\left\{
\begin{aligned}
&\partial_t u^n-\Delta u^n=\mathbb{E}_n\mathcal P(B^n\cdot\nabla B^n-u^n\cdot\nabla u^n),\\
&\partial_tB^n-\Delta B^n=\nabla\times\mathbb{E}_n(v^n\times B^n),\\
&\partial_t{v^n}-\Delta v^n
=\mathbb{E}_n\cP\Bigl(B^n\cdot\nabla B^n-u^n\cdot\nabla u^n-\nabla\times((\nabla\times v^n)\times B^n)\\
&\hspace{5cm}+\nabla\times(v^n\times u^n)+2\nabla\times(v^n\cdot\nabla B^n)\Bigr)\cdotp
\end{aligned}\right.\end{equation}

\noindent\textbf{Third step: uniform estimates} 

We want to apply Proposition \ref{P_41} to our approximate solution $(u^n, B^n, v^n).$ The key point is that since $\mathbb{E}_n$ is an $L^2$ orthogonal projector, it has no effect on the energy estimates. We claim that $T^n$ may be bounded from below by the supremum $T$ of all the times satisfying \eqref{eq:X2}, and that $(u^n, B^n, v^n)_{n\geq1}$ is bounded in $E_{2}(T)$.
To prove our claim, , we split $(u^n, B^n, v^n)$ into 
$$(u^n, B^n, v^n)=(u^{n, L}, B^{n, L}, v^{n, L})+(\wt u^{n}, \wt B^{n}, \wt v^{n}),$$
 where
 $$u^{n, L}:=\mathbb{E}_ne^{t\Delta} u_0,\quad B^{n, L}:=\mathbb{E}_ne^{t\Delta} B_0\andf v^{n, L}:=\mathbb{E}_ne^{t\Delta} v_0.$$ 
 Since  $\mathbb{E}_n$ maps any Besov space $\dot B^s_{2,1}$ to itself with norm $1,$
  Condition \eqref{eq:X2} may be made independent of $n$ and thus, 
  so does the corresponding time $T.$  Now, as $(\wt u^n, \wt B^n, \wt v^n)$ is spectrally supported in $\{\xi\in\mathbb{R}^3\,|\,n^{-1}\leq|\xi|\leq n\}$, the estimate \eqref{4.999} ensures that it belongs to $L^\infty([0,T]; L^2(\mathbb{R}^3))$. So, finally, the standard continuation criterion for ordinary differential equations implies that $T^n$ is greater than any time $T$ satisfying \eqref{eq:X2} and that we have, for all $n\geq 1,$
\begin{eqnarray}
\|(\wt u^n,\wt B^n,\wt v^n)\|_{L^{\infty}_T(\dot B^{\frac{1}{2}}_{2, 1})}\!+\!C_1\|(\wt u^n, \wt B^n, \wt v^n)\|_{L^{1}_T(\dot B^{\frac{5}{2}}_{2, 1})}&\!\!\!\!\leq\!\!\!\!& C\kappa\label{4.888}
\andf\\
\|(u^n, B^n, v^n)\|_{L^{\infty}_T(\dot B^{\frac{1}{2}}_{2, 1})}\!+\!C_1\|(u^n, B^n, v^n)\|_{L^{1}_T(\dot B^{\frac{5}{2}}_{2, 1})}
&\!\!\!\!\leq\!\!\!\!& \|(u_0,B_0,v_0)\|_{\dot B^{\frac{1}{2}}_{2, 1}}+C\kappa.\qquad\quad\label{4.002}
\end{eqnarray}

\noindent\textbf{Fourth step: existence of a solution} 

We claim that, up to an extraction, the sequence $(u^n, B^n, v^n)_{n\in\mathbb{N}}$ converges in $\mathcal{D}'(\mathbb{R}^+\times\mathbb{R}^3)$ to a solution $(u, B, v)$ of \eqref{1.100} supplemented
with data $(u_0,B_0,v_0)$  having the desired regularity properties.  
The definition of $\mathbb{E}_n$ entails that 
\begin{equation*}
(\mathbb{E}_n u_0, \mathbb{E}_n B_0, \mathbb{E}_n v_0)\to(u_0, B_0, v_0)\quad{\rm{in}}\quad\dot{B}^{\frac12}_{2, 1},
\end{equation*}
and Proposition \ref{Le_27} thus ensures  that
$(u^{n, L}, B^{n, L}, v^{n, L})\to (u^L, B^L, v^L)$  in $E_{2}(T).$
\medbreak
Proving  the convergence of $(\wt u^n, \wt B^n, \wt v^n)$ will be achieved 
from  compactness arguments :  we shall exhibit uniform bounds in suitable spaces
for $(\partial_tu^n,\partial_t B^n, \partial_tv^n)_{n\in\N}$ so as to glean some 
H\"older regularity with respect to the time variable. Then, combining
with compact embedding will enable us to  apply Ascoli's theorem and to get the existence of a limit $(u, B, v)$ for a subsequence. Furthermore, the uniform bounds of the previous steps provide us with additional regularity and convergence properties so that we will be able to pass to the limit in \eqref{4.000b}.
Let us start with a  lemma.
\begin{lemma}\label{Le_4.4} Sequence 
$(\wt u^n, \wt B^n, \wt v^n)_{n\geq 1}$ is  bounded in $\cC^{\frac12}([0, T]; \dot{B}^{-\frac12}_{2, 1}).$
\end{lemma}
\begin{proof}
Observe  that $(\wt u^n, \wt B^n, \wt v^n)$  satisfies
\begin{equation}\label{4.000c}
\left\{
\begin{aligned}
&\partial_t \wt u^n=\Delta\wt  u^n+\mathbb{E}_n\mathcal P(B^n\cdot\nabla B^n-u^n\cdot\nabla u^n),\\
&\partial_t\wt  B^n=\Delta\wt  B^n+\nabla\times\mathbb{E}_n (v^n\times B^n),\\
&\partial_t{\wt  v^n}=\Delta\wt  v^n
+\mathbb{E}_n\cP\Bigl(B^n\cdot\nabla B^n-u^n\cdot\nabla u^n-\nabla\times((\nabla\times v^n)\times B^n)\\
&\hspace{5cm}+\nabla\times(v^n\times u^n)+2\nabla\times(v^n\cdot\nabla B^n)\Bigr)\cdotp
\end{aligned}\right.\end{equation}
According to  the uniform bounds \eqref{4.888}, \eqref{4.002} and to the  product laws:
$$\|ab\|_{\dot{B}^{-\frac12}_{2, 1}}\lesssim\|a\|_{\dot{B}^{\frac12}_{2, 1}}\|b\|_{\dot{B}^{\frac12}_{2, 1}}\andf
\|ab\|_{\dot{B}^{\frac12}_{2, 1}}\lesssim\|a\|_{\dot{B}^{\frac12}_{2, 1}}\|b\|_{\dot{B}^{\frac32}_{2, 1}},$$
 the right-hand side of \eqref{4.000c} is uniformly bounded in $L^2_T(\dot{B}^{-\frac12}_{2, 1}).$ 
 Since $\wt u^n(0)=\wt B^n(0)=\wt b^n(0)=0,$  applying H\"older
 inequality completes the proof. \end{proof}

We can now come to the proof of the existence of a solution. Let $(\phi_{j})_{j\in\mathbb{N}}$ be a sequence of $\cC^\infty_0(\mathbb{R}^3)$ cut-off functions supported in the ball $B(0, j+1)$ of $\mathbb{R}^3$ and equal to $1$ in a neighborhood of $B(0, j).$
 Lemma \ref{Le_4.4} tells us that $(\wt u^n, \wt B^n, \wt v^n)_{n\geq 1}$ is uniformly equicontinuous in the space $\cC([0, T]; \dot{B}^{-\frac12}_{2, 1})$ and \eqref{4.888} ensures that 
 it is bounded in $L^\infty([0,T];\dot B^{\frac12}_{2,1}).$ 
  Using the fact that the application $u\mapsto\phi_{j}u$ is compact from $\dot{B}^{\frac12}_{2, 1}$ into 
  $\dot{B}^{-\frac12}_{2, 1},$  combining  Ascoli's theorem and Cantor's diagonal process ensures that there exists some triplet $(\wt u, \wt B, \wt v)$  such that for all $j\in\mathbb{N},$ 
\begin{equation}\label{4.444}
(\phi_{j}\wt u^n, \phi_{j}\wt B^n, \phi_{j}\wt v^n)\to(\phi_{j}\wt u, \phi_{j}\wt B, \phi_{j}\wt v)\quad {\rm{in}}\quad \cC([0, T]; \dot{B}^{-\frac12}_{2, 1}).
\end{equation}
This obviously entails that $(\wt u^n, \wt B^n, \wt v^n)$ tends to $(\wt u, \wt B, \wt v)$ in $\cD'(\mathbb{R}^+\times\mathbb{R}^3).$
 
 Coming back to the uniform estimates of third step and using the argument of \cite[p. 443]{Ba11}
 to justify that there is no time concentration,
 we  get that $(\wt u, \wt B, \wt v)$ belongs to $L^\infty(0,T;\dot B^{\frac12}_{2, 1})\cap L^1(0,T;\dot B^{\frac52}_{2, 1})$ and to  $\cC^{\frac12}([0, T]; \dot{B}^{-\frac12}_{2, 1}).$
 
Let us now prove that $(u, B, v):=(u^L+\wt u, B^L+\wt B, v^L+\wt v)$ solves \eqref{1.100}. The only problem is to pass to the limit in the non-linear terms. 
By way of example,  let us explain how to handle the term $\mathbb{E}_n\cP\nabla\times((\nabla\times v^n)\times B^n)$ in \eqref{4.000b} (actually, $\cP$ may be omitted
as a curl is divergence free). Let $\theta\in\cC^\infty_0(\mathbb{R}^+\times\mathbb{R}^3;\R^3)$ and $j\in\mathbb{N}$ be such that ${\rm{Supp}}\,\theta\subset[0,\,j]\times B(0, j).$ We use the decomposition
$$\displaylines{
\langle\nabla\times\mathbb{E}_n((\nabla\times v^n)\times B^n),\,\theta\rangle-\langle\nabla\times((\nabla\times v)\times B), \theta\rangle\hfill\cr\hfill
=\langle(\nabla\times v^n)\times \phi_j(B^n-B),\,\nabla\times\mathbb{E}_n\theta\rangle
+\langle(\nabla\times\phi_j(v^n-v))\times B,\,\nabla\times\mathbb{E}_n\theta\rangle\hfill\cr\hfill
+\langle\mathbb{E}_n((\nabla\times v)\times B)-(\nabla\times v)\times B,\,\nabla\times\theta\rangle.}$$
As $\nabla\times v^n$ is uniformly bounded in $L^1_T(\dot{B}_{2, 1}^{\frac32})$ and $\phi_j B^n$ tends to $\phi_j B$ in $L^\infty_T(\dot{B}_{2, 1}^{\frac12})$, the first term tends to $0.$  According to the uniform estimates \eqref{4.002} and \eqref{4.444}, $\nabla\times\phi_j(v^n-v)$ tends to $0$ in $L^1_T(\dot{B}_{2, 1}^{\frac12})$ so that the second term tends to $0$ as well. Finally, thanks to \eqref{4.200},  the third term tends to $0.$ 

The other non-linear terms can be treated similarly, and the continuity of $(u, B, v)$ 
stems from Proposition \ref{Le_27}  since the right-hand side of \eqref{1.100} belongs to $L^1_T(\dot{B}^{\frac12}_{2, 1}).$
\medbreak

\noindent\textbf{Fifth step: uniqueness}

Let $(u_1, B_1)$ and $(u_2, B_2)$ be two solutions of the Hall-MHD 
system on $[0,T]\times\R^3,$ with the same initial data,  and such that
$(u_i,B_i,v_i)\in E_2(T)$ for $i=1,2.$  
Then, the difference $(\du, \dB, \dv):=(u_1-u_2, B_1-B_2, v_1-v_2)$ is in  $E_2(T)$ and satisfies
\begin{equation}\label{4.d1}
\left\{\begin{aligned}
 &\partial_t {\du}-\Delta {\du}:=R_1,\\
 &\partial_t {\dB}-\Delta {\dB}:=R_2,\\
 &\partial_t \dv-\Delta\dv:=R_1+R_3+R_4+ R_5,\\
\end{aligned}
\right.
\end{equation}
where
\begin{align*}
&R_1:=\mathcal{P}(B_1\cdot\nabla\dB+\dB\cdot\nabla B_2-u_1\cdot\nabla\du-\du\cdot\nabla u_2),\\
&R_2:=\nabla\times(v_1\times\dB+\dv\times B_2),\\
&R_3:=-\nabla\times((\nabla\times v_1)\times\dB+(\nabla\times\dv)\times B_2),\\
&R_4:=\nabla\times(v_1\times\du+\dv\times u_2),\\
&R_5:=2\nabla\times(v_1\cdot\nabla\dB+\dv\cdot\nabla B_2).
\end{align*}
Hence, arguing as in the first step of the proof gives  for all $t\,\in[0, T],$
\begin{multline}
\|(\du, \dB, \dv)(t)\|_{\dot{B}^{\frac12}_{2, 1}}+\int_0^t\|(\du, \dB, \dv)\|_{\dot{B}^{\frac52}_{2, 1}}\,d\tau\lesssim\int_0^t\Bigl(\|(R_1, R_2, R_4, R_5)\|_{\dot{B}^{\frac12}_{2, 1}}\\
+\|\nabla\times((\nabla\times v_1)\times\dB)\|_{\dot{B}^{\frac12}_{2, 1}}+
\sum_{j\in\mathbb{Z}}2^{\frac{3j}{2}}\|[\ddj, B_2\times](\nabla\times\dv)\|_{L^2}\,\Bigr)d\tau\label{4.555}.
\end{multline}
Putting together the  product laws \eqref{prod1}
and the commutator estimate \eqref{com00} yields
$$\begin{aligned}
\|R_1\|_{\dot{B}^{\frac12}_{2, 1}}&\lesssim\|(u_1, B_1, u_2, B_2)\|_{\dot{B}^{\frac32}_{2, 1}}\|(\du, \dB)\|_{\dot{B}^{\frac32}_{2, 1}},\\
\|R_2\|_{\dot{B}^{\frac12}_{2, 1}}&\lesssim\|(B_2, v_1)\|_{\dot{B}^{\frac32}_{2, 1}}\|(\dB, \dv)\|_{\dot{B}^{\frac32}_{2, 1}},\\
\|R_4\|_{\dot{B}^{\frac12}_{2, 1}}&\lesssim\|(u_2, v_1)\|_{\dot{B}^{\frac32}_{2, 1}}\|(\du, \dv)\|_{\dot{B}^{\frac32}_{2, 1}},\\
\|R_5\|_{\dot{B}^{\frac12}_{2, 1}}&\lesssim\|(\nabla B_2, v_1)\|_{\dot{B}^{\frac32}_{2, 1}}\|(\nabla\dB, \dv)\|_{\dot{B}^{\frac32}_{2, 1}}\\
&\lesssim \|(u_2, v_1, v_2)\|_{\dot{B}^{\frac32}_{2, 1}}\|(\du, \dv)\|_{\dot{B}^{\frac32}_{2, 1}},\end{aligned}$$
$$\begin{aligned}
\|\nabla\times((\nabla\times v_1)\times\dB)\|_{\dot{B}^{\frac12}_{2, 1}}&\lesssim\|\nabla\times v_1\|_{\dot{B}^{\frac32}_{2, 1}}\|\dB\|_{\dot{B}^{\frac32}_{2, 1}}\\
&\lesssim\|v_1\|_{\dot{B}^{\frac52}_{2, 1}}\|(\du, \dv)\|_{\dot{B}^{\frac12}_{2, 1}}
\end{aligned}$$
and
$$\begin{aligned}
\sum_{j\in\mathbb{Z}}2^{\frac{3j}{2}}\|[\ddj, B_2\times](\nabla\times\dv)\|_{L^2}&\lesssim\|\nabla B_2\|_{\dot{B}^{\frac32}_{2, 1}}\|\nabla\times\dv\|_{\dot{B}^{\frac12}_{2, 1}}\\
&\lesssim\|(u_2, v_2)\|_{\dot{B}^{\frac32}_{2, 1}}\|\dv\|_{\dot{B}^{\frac32}_{2, 1}}\cdotp
\end{aligned}$$
Hence, by interpolation and Young's inequality, Inequality \eqref{4.555} becomes
$$\displaylines{
\|(\du, \dB, \dv)(t)\|_{\dot{B}^{\frac12}_{2, 1}}+\int_0^t\|(\du, \dB, \dv)(\tau)\|_{\dot{B}^{\frac52}_{2, 1}}\,d\tau\\
\leq\int_0^tZ(\tau)\|(\du, \dB, \dv)(\tau)\|_{\dot{B}^{\frac12}_{2, 1}}\,d\tau}$$
with $Z(t):=C\bigl(\|(u_1, u_2, B_1, B_2, v_1, v_2)\|_{\dot{B}^{\frac32}_{2, 1}}^2+\|v_1\|_{\dot{B}^{\frac52}_{2, 1}}\bigr)\cdotp$
\medbreak
Thus, Gronwall lemma and our assumptions on the solutions ensure that 
$$(\du, \dB, \dv)\equiv0\quad{\rm{on}}\quad[0, T].$$
\medbreak

\noindent\textbf{Sixth step: Blow-up criterion}

Let us assume that we are given a solution $(u,B)$ on some \emph{finite} time  interval $[0,T^*)$
fulfilling the regularity properties listed in Theorem \ref{Th_2} for all $t<T^*.$ 
Then, applying the method of  the first step to  \eqref{1.100} yields for all $t<T^*,$
\begin{multline}\label{4.2400}
\|(u, B, v)(t)\|_{\dot B^{\frac{1}{2}}_{2, 1}}
+ C_1\int_0^t\|(u, B, v)\|_{\dot B^{\frac{5}{2}}_{2, 1}}\,d\tau
\leq \|(u, B, v)(0)\|_{\dot B^{\frac{1}{2}}_{2, 1}}\\+ \int_0^t\Bigl(\|B\cdot\nabla B\|_{\dot B^{\frac{1}{2}}_{2, 1}}+\|u\cdot\nabla u\|_{\dot B^{\frac{1}{2}}_{2, 1}}+\bigl(\|v\times B\|_{\dot B^{\frac{3}{2}}_{2, 1}}+\|v\times u\|_{\dot B^{\frac{3}{2}}_{2, 1}}\\+\|v\cdot\nabla B\|_{\dot B^{\frac{3}{2}}_{2, 1}}+\sum_j 2^\frac{3j}{2}\|[\dot\Delta_j, B\times](\nabla\times v)\|_{L^2}\bigr)\Bigr)d\tau.\end{multline}
Using the tame estimates \eqref{eq:tame1}, the fact that $\dot{B}^{\frac{3}{2}}_{2, 1}$ is an algebra embedded in $L^\infty,$
interpolation inequalities and  Young's inequality, we get for all $\eta>0,$
\begin{align*}
\|B\cdot \nabla B\|_{\dot{B}^{\frac{1}{2}}_{2, 1}}&\leq C\|B\otimes B\|_{\dot{B}^{\frac{3}{2}}_{2, 1}}\\
&\leq C\|B\|_{L^\infty}\|B\|_{\dot{B}^{\frac{3}{2}}_{2, 1}}\\
&\leq \frac{C}{\eta}\|B\|_{L^\infty}^2\|B\|_{\dot{B}^{\frac{1}{2}}_{2, 1}}+\eta\|B\|_{\dot{B}^{\frac{5}{2}}_{2, 1}},
\end{align*} 
and, similarly, 
$$
\|u\cdot \nabla u\|_{\dot{B}^{\frac{1}{2}}_{2, 1}}
\leq \frac{C}{\eta}\|u\|_{L^\infty}^2\|u\|_{\dot{B}^{\frac{1}{2}}_{2, 1}}+\eta\|u\|_{\dot{B}^{\frac{5}{2}}_{2, 1}}.
$$
We also have 
\begin{align*}
\|v\times B\|_{\dot{B}^{\frac{3}{2}}_{2, 1}}&\leq C(\|v\|_{L^\infty}\|B\|_{\dot{B}^{\frac{3}{2}}_{2, 1}}+\|B\|_{L^\infty}\|v\|_{\dot{B}^{\frac{3}{2}}_{2, 1}})\\
&\leq \frac{C}{\eta}\|(B, v)\|_{L^\infty}^2\|(B, v)\|_{\dot{B}^{\frac{1}{2}}_{2, 1}}+\eta\|(B, v)\|_{\dot{B}^{\frac{5}{2}}_{2, 1}},\\
\|v\times u\|_{\dot{B}^{\frac{3}{2}}_{2, 1}}
&\leq \frac{C}{\eta}\|(u, v)\|_{L^\infty}^2\|(u, v)\|_{\dot{B}^{\frac{1}{2}}_{2, 1}}+\eta\|(u,v)\|_{\dot{B}^{\frac{5}{2}}_{2, 1}},\\
\|v\cdot\nabla B\|_{\dot{B}^{\frac{3}{2}}_{2, 1}}
&\leq \frac{C}{\eta}\|(\nabla B, v)\|_{L^\infty}^2\|(\nabla B, v)\|_{\dot{B}^{\frac{1}{2}}_{2, 1}}+\eta\|(\nabla B,v)\|_{\dot{B}^{\frac{5}{2}}_{2, 1}}.
\end{align*}
As, according to  \eqref{com01} with $s=3/2$ and to the fact that $\nabla: L^\infty\to\dot B^{-1}_{\infty,\infty},$  we have 
\begin{equation}\label{eq:com0}
\sum_j 2^\frac{3j}{2}\|[\dot\Delta_j, B\times](\nabla\times v)\|_{L^2}
\leq C\bigl(\|\nabla B\|_{L^\infty}\|v\|_{\dot{B}^{\frac{3}{2}}_{2, 1}}+\|v\|_{L^\infty}\|\nabla B\|_{\dot B^{\frac32}_{2,1}}\bigr),
\end{equation}
that term  may be bounded as $v\cdot\nabla B.$
\medbreak
Therefore,  if we choose $\eta$ small enough, then \eqref{4.2400} becomes: 
$$\displaylines{
\|(u, B, v)(t)\|_{\dot B^{\frac{1}{2}}_{2, 1}}+ \frac{C_1}2\int_0^t\|(u, B, v)\|_{\dot B^{\frac{5}{2}}_{2, 1}}\,d\tau
\leq  \|(u, B, v)(0)\|_{\dot B^{\frac{1}{2}}_{2, 1}}\hfill\cr\hfill
+C\int_0^t \|(u, B, \nabla B)\|_{L^\infty}^2\|(u, B, v)\|_{\dot B^{\frac{1}{2}}_{2, 1}}\,d\tau}$$
and  Gronwall's inequality implies that  for all $t\in[0,T^*),$
$$\displaylines{
\|(u, B, v)(t)\|_{\dot B^{\frac{1}{2}}_{2, 1}}+ \frac{C_1}2\int_0^t\|(u, B, v)\|_{\dot B^{\frac{5}{2}}_{2, 1}}\,d\tau
\hfill\cr\hfill\leq  \|(u, B, v)(0)\|_{\dot B^{\frac{1}{2}}_{2, 1}}\exp\biggl(C\int_0^t\|(u, B,\nabla B)\|_{L^\infty}^2\,dt\biggr)\cdotp}
$$
Now, if one assumes that 
\begin{equation*}\int_0^{T^*}\|(u, B,\nabla B)(t)\|_{L^\infty}^2\,dt<\infty,\end{equation*}
then the above inequality ensures that $(u,B,v)$ belongs
to $L^\infty(0,T^*;\dot B^{\frac12}_{2,1})$
and one may conclude by classical arguments that the solution may be continued beyond~$T^*.$
\medbreak 
 In order to prove the second blow-up criterion, one uses the following 
 inequalities,  based on \eqref{prod1} and interpolation inequalities:
\begin{align*}
\|B\cdot \nabla B\|_{\dot{B}^{\frac{1}{2}}_{2, 1}}&\lesssim \|B\|_{\dot{B}^{\frac{1}{2}}_{2, 1}}\|B\|_{\dot{B}^{\frac{5}{2}}_{2, 1}},\\
\|u\cdot \nabla u\|_{\dot{B}^{\frac{1}{2}}_{2, 1}}&\lesssim \|u\|_{\dot{B}^{\frac{1}{2}}_{2, 1}}\|u\|_{\dot{B}^{\frac{5}{2}}_{2, 1}},\\
\|v\times B\|_{\dot{B}^{\frac{3}{2}}_{2, 1}}&\lesssim \|v\|_{\dot{B}^{\frac{3}{2}}_{2, 1}}\|B\|_{\dot{B}^{\frac{3}{2}}_{2, 1}}\\
&\lesssim \|v\|_{\dot{B}^{\frac{1}{2}}_{2, 1}}\|v\|_{\dot{B}^{\frac{5}{2}}_{2, 1}}+\|B\|_{\dot{B}^{\frac{1}{2}}_{2, 1}}\|B\|_{\dot{B}^{\frac{5}{2}}_{2, 1}},\\
\|v\times u\|_{\dot{B}^{\frac{3}{2}}_{2, 1}}&\lesssim\|v\|_{\dot{B}^{\frac{3}{2}}_{2, 1}}\|u\|_{\dot{B}^{\frac{3}{2}}_{2,1}}\\
&\lesssim \|v\|_{\dot{B}^{\frac{1}{2}}_{2, 1}}\|v\|_{\dot{B}^{\frac{5}{2}}_{2, 1}}+\|u\|_{\dot{B}^{\frac{1}{2}}_{2, 1}}\|u\|_{\dot{B}^{\frac{5}{2}}_{2, 1}},\\
\|v\cdot\nabla B\|_{\dot{B}^{\frac{3}{2}}_{2, 1}}&\lesssim\|v\|_{\dot{B}^{\frac{3}{2}}_{2, 1}}\|\nabla B\|_{\dot{B}^{\frac{3}{2}}_{2, 1}}\\
&\lesssim \|v\|_{\dot{B}^{\frac{3}{2}}_{2, 1}}^2+\|J\|_{\dot{B}^{\frac{3}{2}}_{2, 1}}^2\\
&\lesssim \|v\|_{\dot{B}^{\frac{3}{2}}_{2, 1}}^2+\|u\|_{\dot{B}^{\frac{3}{2}}_{2, 1}}^2\\
&\lesssim \|v\|_{\dot{B}^{\frac{1}{2}}_{2, 1}}\|v\|_{\dot{B}^{\frac{5}{2}}_{2, 1}}+\|u\|_{\dot{B}^{\frac{1}{2}}_{2, 1}}\|u\|_{\dot{B}^{\frac{5}{2}}_{2, 1}}
\end{align*}
and by \eqref{com00} and  Proposition \ref{P_25} {\it (iii)} {\it (vii)},
$$\begin{aligned}
\sum_j 2^\frac{3j}{2}\|[\dot\Delta_j, B\times](\nabla\times v)\|_{L^2}
&\lesssim \|\nabla B\|_{\dot{B}^{\frac{3}{2}}_{2, 1}}\|v\|_{\dot{B}^{\frac{3}{2}}_{2, 1}}\\
&\lesssim \|v\|_{\dot{B}^{\frac{1}{2}}_{2, 1}}\|v\|_{\dot{B}^{\frac{5}{2}}_{2, 1}}+\|u\|_{\dot{B}^{\frac{1}{2}}_{2, 1}}\|u\|_{\dot{B}^{\frac{5}{2}}_{2, 1}}.\end{aligned}$$
Plugging those estimates in \eqref{4.2400}, we find  that 
$$\displaylines{
\|(u, B, v)(t)\|_{\dot B^{\frac{1}{2}}_{2, 1}}+ C_1\int_0^t\|(u, B, v)\|_{\dot B^{\frac{5}{2}}_{2, 1}}\,d\tau \leq \|(u, B, v)(0)\|_{\dot B^{\frac{1}{2}}_{2, 1}}\hfill\cr\hfill
+\int_0^t\|(u, B, v)\|_{\dot B^{\frac{1}{2}}_{2, 1}}\|(u, B, v)\|_{\dot B^{\frac{5}{2}}_{2, 1}}\,d\tau.}
$$
Hence, if 
\begin{equation*}
\int_0^{T^*}\|(u, B, J)\|_{\dot B^{\frac{5}{2}}_{2, 1}}\,dt <\infty,
\end{equation*}
then the solution may be continued beyond $T^*.$ 
\medbreak

 For proving  the last blow-up criterion,  one can use that for   $\rho\in(2, \infty]$,  most of the terms of \eqref{4.2400} 
 may be bounded by means of Inequality \eqref{prod2}. 
 The last commutator term may be bounded from \eqref{com2} (without time integration)
 with $r=1$ and $s=3/2$ as follows:
$$\begin{aligned}
\sum_j 2^\frac{3j}{2}\|[\dot\Delta_j, B\times](\nabla\times v)\|_{L^2}
&\lesssim \|\nabla B\|_{\dot{B}^{\frac{2}{\rho}-1}_{\infty, \infty}}\|v\|_{\dot{B}^{\frac{5}{2}-\frac{2}{\rho}}_{2, 1}}+\|v\|_{\dot{B}^{\frac{2}{\rho}-1}_{\infty, \infty}}\|\nabla B\|_{\dot{B}^{\frac{5}{2}-\frac{2}{\rho}}_{2, 1}}.\end{aligned}$$
Since, by interpolation, we have 
$$\|Z\|_{\dot{B}^{\frac{5}{2}-\frac{2}{\rho}}_{2, 1}}\lesssim\|Z\|_{\dot{B}^{\frac{1}{2}}_{2, 1}}^{\frac{1}{\rho}}\|Z\|_{\dot{B}^{\frac{5}{2}}_{2, 1}}^{\frac{1}{\rho'}}\with \frac{1}{\rho'}=1-\frac{1}{\rho},$$
 using Young inequality and reverting to \eqref{4.2400} yields
$$\displaylines{
\|(u, B, v)(t)\|_{\dot B^{\frac{1}{2}}_{2, 1}}+\int_0^t\|(u, B, v)\|_{\dot B^{\frac{5}{2}}_{2, 1}}\,d\tau \leq \|(u, B, v)(0)\|_{\dot B^{\frac{1}{2}}_{2, 1}}\hfill\cr\hfill
+C\int_0^t\|(u, B, v)\|_{\dot B^{\frac{2}{\rho}-1}_{\infty, \infty}}^\rho\|(u, B, v)\|_{\dot B^{\frac{1}{2}}_{2, 1}}\,d\tau.}
$$
As before, one can conclude that
if $T^*<\infty$ and \eqref{blowup3} is fulfilled, 
then the solution may be continued beyond $T^*.$ 
This completes  the proof of the theorem.\hfill\quad$\square$



\section{The well-posedness theory  in spaces  \texorpdfstring{$\dot B^{\frac12}_{2,r}$}{TEXT}
for general  \texorpdfstring{$r$}{TEXT}} \label{s:r}

Let us first prove the a priori estimates leading to global existence. 
\begin{proposition}\label{P_5a}
Assume that $(u,B)$ is a smooth solution of the Hall-MHD system on $[0, T]\times\mathbb{R}^3$ with $h=\mu=\nu=1.$ Let $v:=u-\nabla\times B.$ There exists a universal constant $C$ such that for any $r\in[1, \infty]$, we have
\begin{equation}\label{5.00a}
\|(u, B, v)\|_{E_{2, r}(T)}\leq  C\bigl(\|(u_0, B_0, v_0)\|_{\dot{B}^{\frac12}_{2, r}}+\|(u, B, v)\|_{E_{2, r}(T)}^2\bigr)\cdotp
\end{equation}
\end{proposition}
\begin{proof}
We argue as in the proof of Inequality \eqref{4.2400}, but  take
the  $\ell^r(\Z)$ norm instead of the $\ell^1(\Z)$ norm. We get for all $t\in[0,T],$
$$\displaylines{
\|(u,B,v)\|_{\wt L^\infty_t(\dot B^{\frac12}_{2,r})}+
\|(u,B,v)\|_{\wt L^1_t(\dot B^{\frac52}_{2,r})}\lesssim \|(u_0,B_0,v_0)\|_{\dot B^{\frac12}_{2,r}}
+ \|B\cdot\nabla B\|_{\wt L_t^1(\dot B^{\frac12}_{2,r})}  \hfill\cr\hfill
+ \|u\cdot\nabla u\|_{\wt L_t^1(\dot B^{\frac12}_{2,r})} 
 + \|v\cdot\nabla B\|_{\wt L_t^1(\dot B^{\frac12}_{2,r})}  + \|B\cdot\nabla v\|_{\wt L_t^1(\dot B^{\frac12}_{2,r})} 
+ \|v\cdot\nabla u\|_{\wt L_t^1(\dot B^{\frac12}_{2,r})} 
\hfill\cr\hfill + \|u\cdot\nabla v\|_{\wt L_t^1(\dot B^{\frac12}_{2,r})}  
+ \|v\cdot\nabla B\|_{\wt L_t^1(\dot B^{\frac32}_{2,r})}  
+\bigl\|2^{\frac{3j}2}\| [\dot\Delta_j, B\times](\nabla\times v)\|_{L^1_t(L^2)}\bigr\|_{\ell^r(\Z)}.}
$$
The first six nonlinear  terms in the right-hand side may be bounded according to 
the following  product law that is proved in Appendix: 
\begin{equation}\label{5.p1}\|a\,b \|_{\wt L_t^1(\dot B^{\frac12}_{2,r})}  \lesssim \|a\|_{\wt L^{4}_t(\dot B^1_{2,r})}
\|b\|_{\wt L^{\frac43}_t(\dot B^1_{2,r})}.\end{equation}
The last but one term may be bounded as follows: 
\begin{equation}\label{5.p2}
\|v\cdot\nabla B\|_{\wt L_t^1(\dot B^{\frac32}_{2,r})}\lesssim  \|v\|_{\wt L_t^4(\dot B^{1}_{2,r})}\|\nabla B\|_{\wt L^{\frac43}_t(\dot B^2_{2,r})}
+ \|\nabla B\|_{\wt L_t^4(\dot B^{1}_{2,r})}\|v\|_{\wt L^{\frac43}_t(\dot B^2_{2,r})}.\end{equation}
Finally,  in light of  \eqref{com2} with  $b=B$, $a=\nabla\times v,$ $s=3/2$ and $\rho=4,$
 and embedding, one discovers that 
the commutator term may be bounded exactly as $v\cdot\nabla B.$
\medbreak
Putting together all the above inequalities eventually yields for all $t\geq0,$ 
\begin{equation} \label{5.2}
\|(u,B,v)\|_{E_{2,r}(t)}\lesssim  \|(u_0,B_0,v_0)\|_{\dot B^{\frac12}_{2,r}}
+\|(u,B,v)\|_{\wt L_t^{\frac43}(\dot B^2_{2,r})} \|(u,B,v)\|_{\wt L_t^{4}(\dot B^1_{2,r})}. 
\end{equation}
  Since one can prove by making use of H\"older inequality and interpolation that
$$\|z\|_{\wt L^\rho_t(\dot B^{\frac12+\frac2\rho}_{2,r})} \leq \|z\|_{E_{2,r}(t)}\quad\hbox{for all }\ \rho\in[1,+\infty], $$
Inequality \eqref{5.2} implies \eqref{5.00a}.\end{proof}
\bigbreak

In order to prove Theorem \ref{Th_2bis}, we proceed as follows:
\begin{enumerate}
\item   smooth out the data and get a sequence  $(u^n, B^n)_{n\in\N}$
 of global smooth solutions to the Hall-MHD system;
\item  apply Proposition \ref{P_5a} to $(u^n, B^n)_{n\in\N}$ and obtain  uniform estimates for 
$(u^n, B^n, v^n)_{n\in\N}$ in the space $E_{2, r}$;
\item  use compactness to prove that  $(u^n, B^n)_{n\in\N}$ converges, up to extraction, 
to a solution of the Hall-MHD
system supplemented with initial data $(u_0,B_0)$;
\item prove stability estimates \emph{in a larger space} to get the uniqueness of the solution.
\end{enumerate}
\medbreak
To proceed, let us  smooth out the initial data as follows\footnote{The reader
may refer to the appendix for the definition of $\dot S_j$}:
\begin{equation*}
u^n_0 :=(\dot{S}_n-\dot S_{-n}) u_0\andf  B^n_0 :=(\dot{S}_n-\dot S_{-n}) B_0.
\end{equation*}
Clearly, $u^n_0$ and $B^n_0$ belong to all Sobolev spaces, and we have for 
$z=u,B,v$ and all $n\in\N,$
\begin{equation}\label{eq:bounddata}
\forall j\in\mathbb{Z}, \quad\|\ddj z^n_0\|_{L^2}\leq\|\ddj z_0\|_{L^2}\andf 
\|z^n_0\|_{\dot B^{\frac12}_{2,r}}\leq \|z_0\|_{\dot{B}^{\frac12}_{2, r}}.\end{equation}
Since in particular $(u_0^n,B_0^n,v_0^n)$ is in $\dot B^{\frac12}_{2,1},$ Theorem 
\ref{Th_2} guarantees that the Hall-MHD system with data $(u_0^n,B_0^n)$ 
has a unique maximal solution on $[0,T^n)$ for some $T^n>0,$ that belongs
to $E_{2,1}(T)$ for all $T<T^n.$
Now, take some positive real number  $M$ to be chosen later on and define
$$T_n := {\rm{sup}}\bigl\{t\in[0,T^n)\,,\, \|(u^n, B^n, v^n)\|_{E_{2, r}(t)}\leq Mc\bigr\}\cdotp$$
We are going to show first that $T_n=T^n,$ then that $T^n=+\infty.$ 
\smallbreak
According to Proposition \ref{P_5a} and to \eqref{eq:bounddata}, we have 
$$\|(u^n, B^n, v^n)\|_{E_{2, r}(T_n)}\leq C\bigl(\|(u_0,B_0,v_0)\|_{\dot B^{\frac12}_{2,r}}
+ \|(u^n, B^n, v^n)\|_{E_{2, r}(T_n)}^2\bigr)\cdotp$$
Hence, using the smallness condition on $(u_0,B_0,v_0)$ and the definition of $T_n,$
$$\|(u^n, B^n, v^n)\|_{E_{2, r}(T_n)}\leq Cc(1+M^2c).$$
If we take $M=2C,$ then $c$ so that $4C^2c<1,$ then we have
$$\|(u^n, B^n, v^n)\|_{E_{2, r}(T_n)}<Mc,$$
and thus, by a classical continuity argument, $T_n=T^n.$ 
\medbreak
Now, using functional embedding and interpolation arguments, we discover that
$$\biggl(\int_0^{T^n} \!\|(u^n,B^n,v^n)\|^4_{\dot B^{-\frac12}_{\infty,\infty}}dt\biggr)^{\frac14}
\lesssim \|(u^n,B^n,v^n)\|_{\wt L_{T^n}^4(\dot B^1_{2,r})} 
\lesssim \|(u^n, B^n, v^n)\|_{E_{2, r}(T^n)}.$$
Hence, the continuation criterion \eqref{blowup3} guarantees that, indeed, $T^n=+\infty.$
This means that the solution is global and that, furthermore, 
\begin{equation}\label{5.220}
\|(u^n, B^n, v^n)\|_{E_{2, r}}\leq Mc\quad\hbox{for all }\ n\in\mathbb{N}.\end{equation} 
At this stage, proving that $(u^n,B^n)_{n\in\N}$ converges (up to subsequence)  to a global solution $(u,B)$ of the Hall-MHD system with data $(u_0,B_0)$
and $(u,B,v)$ in $E_{2,r}$  follows from the same arguments 
as in the previous section. 
\medbreak 
Let us finally prove the uniqueness part of the theorem. Suppose that $(u_1, B_1)$ and $(u_2, B_2)$ are two solutions of the Hall-MHD system on $[0, T]\times\mathbb{R}^3$ supplemented with the same initial data $(u_0, B_0)$ and such that
$$(u_i,B_i,v_i)\in \wt\cC([0,T];\dot B^{\frac12}_{2, r})\cap \wt L^1(0,T;\dot B^{\frac52}_{2, r}),\quad i=1,2.$$  
In order to prove the uniqueness, we look at  the difference 
$(\du, \dB, \dv)=(u_1-u_2, B_1-B_2, v_1-v_2)$ as a solution of  System \eqref{4.d1}.
In contrast with the previous section however, we do not
know how to estimate the difference  in the space $E_{2, r}(T)$ since 
  the term $\nabla\times((\nabla\times v_1)\times\dB)$ 
cannot be bounded in the space $\wt L_T^1(\dot B^{\frac12}_{2,r})$ from 
the norm of $v_1$ and $\dB$ in  $E_{2, r}(T)$  (this is  
due to the fact that the norm of  $E_{2, r}(T)$ fails to  control $\|\cdot\|_{L^\infty(0,T\times\R^3)}$ 
by a little if $r>1$). 

For that reason, we shall accept to lose some regularity in the stability estimates
and prove uniqueness in the space
 $$F_{2, r}(T):=\wt L^\infty_T(\dot{B}^{-\frac12}_{2, r}).$$
We need first to justify  that $(\du,\dB,\dv)$  belongs to that space, though.
According  to Proposition  \ref{Le_27}, 
it is enough to check that the terms  $R_1$ to $R_5$ defined 
just below \eqref{4.d1} belong to  $\wt L^1_T(\dot{B}^{-\frac12}_{2, r}).$
Now, from  \eqref{5.p1} and Holder inequality, we have
\begin{align*}
&\|R_1\|_{\wt L^1_T(\dot{B}^{-\frac12}_{2, r})}\lesssim T^{\frac12}\|(u_1, B_1, u_2, B_2)\|_{\wt L^4_T(\dot{B}^{1}_{2, r})}\|(\du, \dB)\|_{\wt L^4_T(\dot{B}^{1}_{2, r})},\\
&\|R_2\|_{\wt L^1_T(\dot{B}^{-\frac12}_{2, r})}\lesssim T^{\frac12}\|(B_2, v_1)\|_{\wt L^4_T(\dot{B}^{1}_{2, r})}\|(\dB, \dv)\|_{\wt L^4_T(\dot{B}^{1}_{2, r})},\\
&\|R_3\|_{\wt L^1_T(\dot{B}^{-\frac12}_{2, r})}\lesssim\|(\nabla v_1, \nabla\dv)\|_{\wt L^\frac43_T(\dot{B}^{1}_{2, r})}\|(\dB, B_2)\|_{\wt L^4_T(\dot{B}^{1}_{2, r})},\\
&\|R_4\|_{\wt L^1_T(\dot{B}^{-\frac12}_{2, r})}\lesssim T^{\frac12}\|(u_2, v_1)\|_{\wt L^4_T(\dot{B}^{1}_{2, r})}\|(\du, \dv)\|_{\wt L^4_T(\dot{B}^{1}_{2, r})},\\
&\|R_5\|_{\wt L^1_T(\dot{B}^{-\frac12}_{2, r})}\lesssim\|(\nabla B_2, \nabla\dB)\|_{\wt L^\frac43_T(\dot{B}^{1}_{2, r})}\|(\dv, v_1)\|_{\wt L^4_T(\dot{B}^{1}_{2, r})}.
\end{align*}
Since the norm in  $E_{2, r}(T)$ bounds the norm  in  $\wt L^4_T(\dot{B}^{1}_{2, r})\cap\wt L^\frac43_T(\dot{B}^{2}_{2, r}),$  one can indeed conclude that  the  terms  $R_1$ to $R_5$  are in $\wt L^1_T(\dot{B}^{-\frac12}_{2, r}).$
\medbreak
 Next,  estimating  $(\du, \dB, \dv)$ in  $F_{2, r}(T)$
 may be achieved by a slight modification of the beginning of the proof of 
 Proposition \ref{P_5a}. We get for all $t\in[0, T],$
\begin{align*}
&\|(\du, \dB, \dv)\|_{F_{2, r}(t)} 
\lesssim\|B_1\cdot\nabla\dB\|_{\wt L_t^1(\dot B^{-\frac12}_{2,r})}
+\|\dB\cdot\nabla B_2\|_{\wt L_t^1(\dot B^{-\frac12}_{2,r})} \\
 &\hspace{1cm}+\|u_1\cdot\nabla\du\|_{\wt L_t^1(\dot B^{-\frac12}_{2,r})}  
+ \|\du\cdot\nabla u_2\|_{\wt L_t^1(\dot B^{-\frac12}_{2,r})}
+ \|v_1\cdot\nabla\dB\|_{\wt L_t^1(\dot B^{-\frac12}_{2,r})}\\
  &\hspace{1cm}+\|\dB\cdot\nabla v_1\|_{\wt L_t^1(\dot B^{-\frac12}_{2,r})}
 + \|B_2\cdot\nabla\dv\|_{\wt L_t^1(\dot B^{-\frac12}_{2,r})} 
+ \|\dv\cdot\nabla B_2\|_{\wt L_t^1(\dot B^{-\frac12}_{2,r})} \\
&\hspace{1cm}+ \|v_1\cdot\nabla\du\|_{\wt L_t^1(\dot B^{-\frac12}_{2,r})}  
+ \|\du\cdot\nabla v_1\|_{\wt L_t^1(\dot B^{-\frac12}_{2,r})}  
+ \|u_2\cdot\nabla\dv\|_{\wt L_t^1(\dot B^{-\frac12}_{2,r})} \\
&\hspace{1cm}+ \|\dv\cdot\nabla u_2\|_{\wt L_t^1(\dot B^{-\frac12}_{2,r})}  
+\|(\nabla\times v_1)\times\dB\|_{\wt L_t^1(\dot B^{\frac12}_{2,r})}  
+\|v_1\cdot\nabla\dB\|_{\wt L_t^1(\dot B^{\frac12}_{2,r})}  \\
&\hspace{1cm}+\|\dv\cdot\nabla B_2\|_{\wt L_t^1(\dot B^{\frac12}_{2,r})}  
+\bigl\|2^{\frac{j}2}\| [\dot\Delta_j, B_2\times](\nabla\times \dv)\|_{L^1_t(L^2)}\bigr\|_{\ell^r(\Z)}.
\end{align*}
Most of the terms on the right-hand side can be bounded by means of  the following
inequalities that are proved in appendix:
\begin{equation}\label{5.p11}
\|a b\|_{\wt L_t^1(\dot B^{-\frac12}_{2, r})}\lesssim\|a\|_{\wt L_t^4(\dot B^{1}_{2,r})}\|b\|_{\wt L_t^\frac43(\dot B^{0}_{2,r})},
\end{equation}
\begin{equation}\label{5.p22}
\|a b\|_{\wt L_t^1(\dot B^{-\frac12}_{2, r})}\lesssim\|a\|_{\wt L_t^\frac43(\dot B^{1}_{2,r})}\|b\|_{\wt L_t^4(\dot B^{0}_{2,r})}.
\end{equation}
Next,  owing  to Inequality \eqref{5.p1} and interpolation, we have
$$\begin{aligned}
\|(\nabla\times v_1)\times\dB&\|_{\wt L_t^1(\dot B^{\frac12}_{2,r})}
+\|v_1\cdot\nabla\dB\|_{\wt L_t^1(\dot B^{\frac12}_{2,r})}  
+\|\dv\cdot\nabla B_2\|_{\wt L_t^1(\dot B^{\frac12}_{2,r})}\\
&\lesssim \|\dB\|_{\wt L_t^4(\dot B^{1}_{2,r})}\|\nabla v_1\|_{\wt L_t^{\frac43}(\dot B^{1}_{2,r})}+\|\nabla\dB\|_{\wt L_t^{\frac43}(\dot B^{1}_{2,r})}\|v_1\|_{\wt L_t^{4}(\dot B^{1}_{2,r})}\\
&\hspace{5.6cm}+\|\dv\|_{\wt L_t^{\frac43}(\dot B^{1}_{2,r})}\|\nabla B_2\|_{\wt L_t^{4}(\dot B^{1}_{2,r})}\\
&\lesssim \bigl(\|(u_2, v_2, v_1)\|_{\wt L_t^{4}(\dot B^{1}_{2,r})}+\|v_1\|_{\wt L_t^{\frac43}(\dot B^{2}_{2,r})}\bigr)\|(\du, \dv)\|_{F_{2, r}(t)}.
\end{aligned}$$
Finally, applying \eqref{com2} with $\rho=4,$ $s=1/2$ and using 
 embeddings $\dot B^0_{2,r}\hookrightarrow \dot B^{-\frac32}_{\infty,\infty}$
and $\dot B^1_{2,r}\hookrightarrow \dot B^{-\frac12}_{\infty,\infty}$ yields
\begin{align*}
\bigl\|2^{\frac{j}2}\| [\dot\Delta_j, B_2\times](\nabla\times \dv)\|_{L^1_t(L^2)}\bigr\|_{\ell^r(\Z)}
\lesssim \|\nabla B_2\|_{L_t^4(\dot B^{1}_{2, r})} \|\nabla\times\dv\|_{\wt L_t^{\frac43}(\dot B^{0}_{2,r})}.
\end{align*}
Thus, one can conclude that
\begin{align*}
&\hspace{1cm}\|(\du, \dB, \dv)\|_{F_{2, r}(t)}\leq Y(t)\|(\du, \dB, \dv)\|_{F_{2, r}(t)}
\end{align*}
with $Y(t) :=\sum_{i=1, 2}\|(u_i, B_i, v_i)\|_{\wt L_t^{4}(\dot B^{1}_{2,r})}+\|v_1\|_{\wt L_t^{\frac43}(\dot B^{2}_{2,r})}.$

\noindent Now, Lebesgue dominated convergence theorem ensures that $Y$ is a continuous nondecreasing function vanishing at zero. Hence $(\du, \dB, \dv)\equiv 0$ in $\wt L^\infty_t(\dot{B}^{-\frac12}_{2, r})\cap\wt L^1_t(\dot{B}^{\frac32}_{2, r})$ for small enough $t$. Combining with a standard connectivity argument 
allows to conclude that  $(\du, \dB, \dv)\equiv 0$ on $\mathbb{R}^+.$
This completes the proof of the theorem in the small data case.\hfill\quad$\square$
\medbreak
Let us briefly explain how the above arguments have 
to be modified so as to handle the case where only $v_0$ is small. 
Note that no smallness condition is needed whatsoever in the proof of uniqueness.
As regards the existence part, we split $u$ and $B$ (not $v$) into 
$u=u^L+\wt u$ and $B=B^L+\wt B$ and repeat the proof of Proposition \ref{P_5a}
on the system fulfilled by $(\wt u,\wt B,v)$ rather than \eqref{1.100}. 
Instead of \eqref{5.2}, we get
 $$\|(\wt u,\wt B,v)\|_{E_{2,r}(t)}\lesssim  \|v_0\|_{\dot B^{\frac12}_{2,r}}
+\|(u,B,v)\|_{\wt L_t^{\frac43}(\dot B^2_{2,r})} \|(u,B,v)\|_{\wt L_t^{4}(\dot B^1_{2,r})}$$ 
 from which we deduce that 
  $$\displaylines{\|(\wt u,\wt B,v)\|_{E_{2,r}(t)}\lesssim  \|v_0\|_{\dot B^{\frac12}_{2,r}}
  +\|(u^L,B^L)\|_{\wt L_t^{\frac43}(\dot B^2_{2,r})} \|(u^L,B^L)\|_{\wt L_t^{4}(\dot B^1_{2,r})}
  \hfill\cr\hfill+ \|(u^L,B^L)\|_{\wt L_t^{\frac43}(\dot B^2_{2,r})\cap \wt L_t^{4}(\dot B^1_{2,r})}
  \|(\wt u,\wt B,v)\|_{E_{2,r}(t)}  +\|(\wt u,\wt B,v)\|_{E_{2,r}(t)}^2. }$$ 
Since, by dominated convergence theorem, we have  
$$\lim_{t\to0}\,\bigl( \|(u^L,B^L)\|_{\wt L_t^{\frac43}(\dot B^2_{2,r})} +\|(u^L,B^L)\|_{\wt L_t^{4}(\dot B^1_{2,r})}\bigr)=0,$$
it is easy to see that if $\|v_0\|_{\dot B^{\frac12}_{2,r}}$ is small enough, then 
one can get a control on $\|(\wt u,\wt B,v)\|_{E_{2,r}(t)}$ for small enough $t.$
From this, repeating essentially the same arguments as in the small data case,
one gets a local-in-time existence statement.


\appendix
\section{Besov Spaces and commutator estimates}
Here, we briefly recall the definition of the Littlewood-Paley decomposition, define Besov spaces
and list  some properties  that have been used repeatedly in the paper. 
For the reader's convenience, we also prove some nonlinear and commutator estimates. 
More details and proofs may be found in  e.g.  \cite{Ba11}.

\medbreak

 The Littlewood-Paley decomposition is a dyadic localization procedure in the frequency space for tempered distributions over $\mathbb R^{d}$. 
 To  define it, fix some nonincreasing smooth radial function  $\chi$   on $\R^d,$ supported 
in (say) $B(0,4/3)$ and with value $1$ on $B(0,3/4),$  and 
set  $\varphi(\xi):=\chi(\xi/2)-\chi(\xi).$  Then, we have 
$$\forall \xi \in \mathbb{R}^{d}, ~\chi (\xi) + \sum_{j\geq0}\varphi(2^{-j}\xi) = 1\andf
\forall \xi \in \mathbb{R}^{d}\setminus\{0\}, ~\sum_{j\in \mathbb{Z} }\varphi(2^{-j}\xi) = 1.
$$

The homogeneous dyadic blocks $\dot\Delta_{j}$ and  low-frequency cut-off operator $\dot S_{j}$ are defined for all $j\in\mathbb Z$ by
$$\begin{aligned}
\dot\Delta_{j}u~:=\varphi(2^{-j}D)u=2^{jd}\int_{\mathbb R^{d}}~h(2^{j}y)u(x-y)\,dy\with h{:=}\mathcal{F}^{-1}\varphi, \\
\dot S_{j}u~:=\chi (2^{-j}D)u=2^{jd}\int_{\mathbb R^{d}}~\tilde{h}(2^{j}y)u(x-y)\,dy\with \tilde{h}{:=}\mathcal{F}^{-1}\chi.
\end{aligned}$$
 The following  \emph{Littlewood-Paley decomposition} of $u$:
\begin{equation*}
 u=\sum_{j\in\mathbb Z}\dot\Delta_{j}u
\end{equation*}
 holds true modulo polynomials for any tempered distribution $u$.
 In order to have an equality in the sense of  tempered
 distributions, we  consider only elements of  the set $\mathcal{S}^{'}_{h}(\mathbb R^{d})$ of tempered distributions $u$ such that
\begin{equation*}
\underset{j\to-\infty}{\lim}~\|\dot S_{j}u\|_{L^{\infty}}=0.
\end{equation*}

\begin{definition}
Let $s$ be a real number and $(p, r)$ be in $[1,\infty]^{2}$. The  homogeneous Besov space $\dot B^{s}_{p, r}$ is the set of distributions $u$ in $\mathcal{S}^{'}_{h}$ such that
\begin{equation*}
\|u\|_{\dot B^{s}_{p, r}}{:=}\|2^{js}\|\dot\Delta_{j}u\|_{L^{p}(\mathbb{R}^d)}\|_{\ell^r(\mathbb{Z})}<\infty.
\end{equation*}
\end{definition}

\begin{proposition}\label{P_25}
The following properties hold true:
\begin{itemize}
\item[(i)]  Derivatives: for all $s\in\mathbb R$~~and~~$1\leq p, r\leq\infty,$ 
we have 
\begin{equation*}
\sup_{|\alpha|=k} \|\partial^{\alpha}u\|_{\dot B^{s}_{p, r}}\sim\|u\|_{\dot B^{s+k}_{p, r}}.
\end{equation*}
\item[(ii)] Embedding: we have the following continuous embedding
$$\dot B^{s}_{p, r}\hookrightarrow \dot B^{s-d(\frac{1}{p}-\frac{1}{\tilde{p}})}_{\tilde{p}, \tilde{r}}\quad{\rm{ whenever }} \ \tilde{p}\geq p
\andf  \tilde{r}\geq r,$$
and the space $\dot B^{\frac dp}_{p,1}$ is embedded in the set
of bounded continuous functions.
\item[(iii)]  Real interpolation: for any~$\theta\in(0, 1)$ and $\,s<\tilde s,$ we have
$$\|u\|_{\dot B^{\theta s+(1-\theta)\tilde{s}}_{p, 1}}\lesssim\|u\|_{\dot B^{s}_{p, \infty}}^{\theta}\|u\|_{\dot B^{\tilde{s}}_{p, \infty}}^{1-\theta}.$$
\item[(iv)]  Completeness: the space $\dot B^s_{p,r}$ is complete if (and only if)
$(s,p,r)$ satisfies 
\begin{equation}\label{eq:cond}
s<\frac{d}{p},\quad{\rm{or}} ~~~s=\frac{d}{p}\andf r=1.
\end{equation}
\item[(v)]  Density: the space $\cS_0(\R^d)$ of Schwartz functions on $\R^d$
with Fourier transform supported away from the origin is dense in $\dot B^s_{p,r}$ 
whenever both $p$ and $r$ are finite.
\item[(vi)] Scaling invariance: for any $s\in\R$ and $(p,r)\in[1,\infty]^2,$ there exists a constant $C$ such that
for all positive $\lambda$ and $u\in\dot B^s_{p,r},$ we have
$$
C^{-1}\lambda^{s-\frac dp} \|u\|_{\dot B^s_{p,r}}\leq 
\|u(\lambda\cdot)\|_{\dot B^s_{p,r}}\leq C\lambda^{s-\frac dp}  \|u\|_{\dot B^s_{p,r}}.
$$
\item[(vii)]  Let $f$ be a smooth function on $\mathbb{R}^{d}\setminus\{0\}$ which is homogeneous of degree 0.  Define $f(D)$ on $\cS(\R^d)$  by \begin{equation*}
\mathcal{F}(f(D)u)(\xi){:=}f(\xi)\mathcal{F}u(\xi),
\end{equation*}
Then,  for all exponents $(s,p,r),$ we have the estimate
$$\|f(D) u\|_{\dot B^s_{p,r}}\lesssim \|u\|_{\dot B^s_{p,r}}.$$
If  in addition $f(D)$  extends to  a map  from ${\mathcal S}'_h(\R^d)$ to itself
and  \eqref{eq:cond} is fulfilled, then $f(D)$ is continuous from $\dot B^{s}_{p, r}$ to $\dot B^{s}_{p, r}.$ 
 \item[(viii)]  
Operator ${\rm curl}^{-1}$ maps $\dot B^{s-1}_{p,1}$ to $\dot B^s_{p,1}$ if $1\leq p<\infty$ and $s\leq d/p.$
\end{itemize}
\end{proposition}
\begin{proof} We only prove  the last  item as it is fundamental in our analysis. 
  Owing to the definition  in 
\eqref{eq:curl-1}, it is obvious that  ${\rm curl}^{-1}$ maps  $\cS_0(\R^d)$ to itself, and 
 homogeneity of degree $-1$ implies that  
we have for all $u$ in $\cS_0(\R^d)$:
$$\| {\rm curl}^{-1} u\|_{\dot B^s_{p,1}}\lesssim \|u\|_{\dot B^{s-1}_{p,1}}.$$
As $\cS_0(\R^d)$ is dense in $\dot B^{s-1}_{p,1}$ and since the space $\dot B^s_{p,1}$ is complete
(owing to $s\leq d/p$),  we get the result. 
\end{proof}

A great deal of our analysis relies on regularity estimates for the  heat equation:
$$ \left\{\begin{aligned}
 &\partial_{t}u-\Delta u = f, \\
&u|_{t=0} = u_{0} .\end{aligned}
\right.\leqno(H)$$
It is classical that for all 
  $u_{0}\in \mathcal{S}'(\mathbb R^{d})$ and $f\in L^1_{loc}(\mathbb R^{+}; \mathcal{S}'(\mathbb R^{d})),$ equation $(H)$ has a unique tempered distribution solution, 
  given by  the following Duhamel formula:
\begin{equation}
u(t)=e^{t\Delta}u_{0}+\int_{0}^{t} e^{(t-\tau)\Delta}f(\tau)~d\tau,\qquad t\geq0.\label{2.1}
\end{equation}
Above,    $(e^{t\Delta})_{t\geq0}$ stands for the heat semi-group. 
It is defined on $\mathcal{S}(\R^d)$ by 
\begin{equation}\label{eq:heat}
\mathcal{F}(e^{t\Delta}z)(\xi) := e^{-t|\xi|^{2}}\widehat z(\xi),
\end{equation}
and is  extended to  the set of tempered distributions by duality. 
\medbreak
As observed  by Chemin  in \cite{Ch99}, the following spaces are suitable for describing the maximal regularity properties of the heat equation. 
\begin{definition}\label{d:tilde}
\noindent For $T>0$, $s\in\mathbb R$, $1\leq\rho\leq\infty$, we set
\begin{equation*}
\|u\|_{\widetilde{L}^{\rho}_{T}(\dot B^{s}_{p, r})} :~=\bigl\|2^{js}\|\dot\Delta _{j}u\|_{{L}^{\rho}_{T}(L^{p})}\bigr\|_{\ell^r(\Z)}.
\end{equation*}
\noindent We define the space $\widetilde{L}^{\rho}_{T}(\dot B^{s}_{p, r})$ to be  the set of tempered distribution $u$ on $(0, T)\times\mathbb R^{d}$ such that $\underset{j\to-\infty}{\lim}~\|\dot S_{j}u(t)\|_{L^\infty}=0$ a.e. in $(0,T),$ and $\|u\|_{\widetilde{L}^{\rho}_{T}(\dot B^{s}_{p, r})}<\infty$.
The space $\widetilde{L}^{\rho}_{T}(\dot B^{s}_{p, r})\cap\cC([0,T];\dot B^s_{p,r})$ is
denoted by $\wt\cC_T(\dot B^s_{p,r}).$ 
In the case $T=+\infty,$ one denotes the corresponding space and norm
by  $\widetilde{L}^{\rho}(\dot B^{s}_{p, r})$ and  $\|\cdot\|_{\widetilde{L}^{\rho}(\dot B^{s}_{p, r})},$
respectively.
\end{definition}
The above spaces or norms may be compared to  more classical ones
according to  Minkowski's inequality:
\begin{equation*}
\|u\|_{\widetilde{L}^{\rho}_{T}(\dot B^{s}_{p, r})}\leq\|u\|_{{L}^{\rho}_{T}(\dot B^{s}_{p, r})}~~{\rm{if}}~~r\geq\rho\andf
\|u\|_{\widetilde{L}^{\rho}_{T}(\dot B^{s}_{p, r})}\geq\|u\|_{{L}^{\rho}_{T}(\dot B^{s}_{p, r})}~~{\rm{if}}~~r\leq\rho.
\end{equation*}

The following fundamental result has been  proved in  \cite{Ch99}.
\begin{proposition}\label{Le_27}
Let $T>0$, $s\in\mathbb{R}$ and $1\leq\rho, p, r\leq\infty$. Assume that $u_{0}\in\dot B^{s}_{p, r}$ and $f\in\widetilde{L}^{\rho}_{T}(\dot B^{s-2+\frac{2}{\rho}}_{p, r})$. Then, $(H)$ has a unique solution $u$ in $\widetilde{L}^{\rho}_{T}(\dot B^{s+\frac{2}{\rho}}_{p, r})\cap\widetilde{L}^{\infty}_{T}(\dot B^{s}_{p, r})$ and there exists a constant $C$ depending only on $d$ and such that for all $\rho_{1}\in[\rho, \infty]$, we have
\begin{equation}\label{2.2}
\|u\|_{\widetilde{L}^{\rho_{1}}_{T}(\dot B^{s+\frac{2}{\rho_{1}}}_{p, r})}\leq C(\|u_{0}\|_{\dot B^{s}_{p, r}} + \|f\|_{\widetilde{L}^{\rho}_{T}(\dot B^{s-2+\frac{2}{\rho}}_{p, r})}).
\end{equation}
Furthermore,  if  $r$ is finite, then $u$ belongs to $\mathcal{C}([0,T]; \dot B^{s}_{p, r})$.
\end{proposition}

Let us now recall a few nonlinear estimates in Besov spaces, that we used in the paper. 
They all may be easily
proved by using the following so-called Bony decomposition (from \cite{Bo81}) 
for the (formal)  product of two distributions $u$ and $v$:
$$u\,v=T_uv+T_vu+R(u, v).$$
Above, $T$ designates the paraproduct bilinear operator defined by 
$$T_uv := \sum_{j}\dot{S}_{j-1}u\ddj v,\quad T_{v}u := \sum_{j}\dot{S}_{j-1}v\,\ddj u$$
and $R$ stands for the remainder operator given by 
$$R(u, v) :=\sum_{j}\sum_{|j'-j|\leq 1}\ddj u\,\dot{\Delta}_{j'}v.$$
The following properties   of  the paraproduct and remainder operators are classical:
\begin{proposition}\label{P_para}
For any $(s, p, r)\in\mathbb{R}\times[1, \infty]^2$ and $t<0$, there exists a constant $C$ such that
$$\|T_uv\|_{\dot{B}^{s}_{p, r}}\leq C\,\|u\|_{L^\infty}\|v\|_{\dot{B}^s_{p, r}}\quad{\rm{and}}\quad \|T_uv\|_{\dot{B}^{s+t}_{p, r}}\leq C\,\|u\|_{\dot{B}^{t}_{\infty, \infty}}\|v\|_{\dot{B}^s_{p, r}}\cdotp$$
For any $(s_1, p_1, r_1)$ and $(s_2, p_2, r_2)$ in $\mathbb{R}\times[1, \infty]^2$  satisfying
$$s_1+s_2>0,\quad \frac{1}{p} :=\frac{1}{p_1}+\frac{1}{p_2}\leq 1\andf\frac{1}{r} :=\frac{1}{r_1}+\frac{1}{r_2}\leq 1,$$ there exists a constant $C$ such that 
$$\|R(u, v)\|_{\dot{B}^{s_1+s_2}_{p, r}}\leq C\,\|u\|_{\dot{B}^{s_1}_{p_1, r_1}}\|v\|_{\dot{B}^{s_2}_{p_2, r_2}}.$$
\end{proposition}
Combining the above proposition with the Bony decomposition allows to get 
a number  of inequalities like, for instance:
\begin{itemize}
\item tame estimates:  for any $s>0$ and $1\leq p,r\leq\infty,$
\begin{equation}\label{eq:tame1}
\|uv\|_{\dot B^s_{p,r}}\lesssim \|u\|_{L^\infty} \|v\|_{\dot B^s_{p,r}}
+  \|v\|_{L^\infty} \|u\|_{\dot B^s_{p,r}};\end{equation}
\item the following product estimate: 
\begin{equation}\label{prod1}
\|uv\|_{\dot B^{s_{1}+s_{2}-\frac{d}{p}}_{p, 1}}\lesssim\|u\|_{\dot B^{s_{1}}_{p, 1}}\|v\|_{\dot B^{s_{2}}_{p,  1}}
\end{equation}
whenever\footnote{In particular, $\dot B^{\frac{d}{p}}_{p, 1}$ is an algebra for any $1\leq p<\infty.$}  $s_{1}, s_{2}\leq\frac{d}{p}$~satisfy $s_{1}+s_{2}>d\max(0,\frac 2p-1)$; 
\item  the following inequality (in the case $d=3$ and $\rho>2$)
that has been used in the proof of \eqref{blowup3}:
\begin{equation}\label{prod2}
\|ab\|_{\dot B^{\frac{3}{2}}_{2, 1}}\lesssim \|a\|_{\dot B^{\frac{2}{\rho}-1}_{\infty, \infty}}\|b\|_{\dot B^{\frac52-\frac{2}{\rho}}_{2, 1}}+ \|b\|_{\dot B^{\frac{2}{\rho}-1}_{\infty, \infty}}\|a\|_{\dot B^{\frac52-\frac{2}{\rho}}_{2, 1}}.
\end{equation}
\end{itemize}

\begin{remark}\label{Re_3} Proposition \ref{P_para} and 
estimates like \eqref{prod1} or \eqref{prod2} may be adapted  to the spaces $\widetilde{L}^{\rho}_{T}(\dot B^{s}_{p, r})$. The general principle is that the time exponent behaves according to H\"older inequality.
 For example, we have
$$\|T_ab\|_{\wt L_t^1(\dot B^{\frac12}_{2,r})}+\|T_ba\|_{\wt L_t^1(\dot B^{\frac12}_{2,r})}+\|R(a, b)\|_{\wt L_t^1(\dot B^{\frac12}_{2,r})}\lesssim \|a\|_{\wt L^{\frac43}_t(\dot B^{-\frac12}_{\infty,\infty})}
\|b\|_{\wt L^{4}_t(\dot B^1_{2,r})}.$$
Then, combining with embedding (in the case $d=3$) gives Inequality \eqref{5.p1}. 
\medbreak
Similarly,   Inequality \eqref{5.p2} stems from
$$\displaylines{
\|T_ab\|_{\wt L^1_t(\dot B^{\frac32}_{2,r})}+\|T_ba\|_{\wt L_t^1(\dot B^{\frac32}_{2,r})}+\|R(a, b)\|_{\wt L_t^1(\dot B^{\frac32}_{2,r})}\hfill\cr\hfill
\lesssim\|a\|_{\wt L^{4}_t(\dot B^{-\frac12}_{\infty, \infty})}
\|b\|_{\wt L^{\frac43}_t(\dot B^2_{2,r})}+\|b\|_{\wt L^{4}_t(\dot B^{-\frac12}_{\infty, \infty})}\|a\|_{\wt L^{\frac43}_t(\dot B^2_{2,r})}.}$$
In order to prove Inequality \eqref{5.p11}, it suffices to use the fact that 
$$\begin{aligned}
\|T_ab\|_{\wt L_t^1(\dot{B}^{-\frac12}_{2, r})}&\lesssim\|a\|_{\wt L_t^4(\dot{B}^{-\frac12}_{\infty, \infty})}
\| b\|_{\wt L^{\frac43}_t(\dot{B}^{0}_{2, r})},\\
\|T_{b} a\|_{\wt L_t^1(\dot{B}^{-\frac12}_{2, r})}&\lesssim\|b\|_{\wt L_t^{\frac43}(\dot{B}^{-\frac32}_{\infty, \infty)}}\|a\|_{\wt L_t^4(\dot{B}^{1}_{2, r})},\\
\|R(a, b)\|_{\wt L_t^1(\dot{B}^{-\frac12}_{2, r})}
&\lesssim\|a\|_{\wt L_t^4(\dot{B}^{1}_{2, r})}\|b\|_{\wt L_t^{\frac43}(\dot{B}^{0}_{2, r})}.\end{aligned}$$
Proving  Inequality \eqref{5.p22} is similar.
\end{remark}

We end this appendix with the proof of commutator estimates
that were crucial in our analysis. 
\begin{proposition} Let $s$ be in $(0,d/2].$ 
Then, we have:
\begin{equation}\label{com01}
\sum_{j\in\Z} 2^{js}\|[\dot\Delta_j, b]a\|_{L^2}
\lesssim\|\nabla b\|_{L^\infty}\|a\|_{\dot{B}^{s-1}_{2, 1}}+\|a\|_{\dot B^{-1}_{\infty,\infty}}
\|\nabla b\|_{\dot B^{s}_{2,1}}.
\end{equation}
Furthermore, 
 for all $r\in[1, \infty]$  and    $\rho\in(2, \infty],$ we have if we set   $\,1/\rho' :=1-1/\rho,$
\begin{multline}\label{com2}
\bigl\|2^{js}\| [\dot\Delta_j, b]a\|_{L^1_t(L^2)}\bigr\|_{\ell^r(\Z)}
\lesssim \|\nabla b\|_{\wt L_t^{\rho}(\dot B^{\frac{2}{\rho}-1}_{\infty, \infty})} \|a\|_{\wt L^{\rho'}_t(\dot B^{s-\frac{2}{\rho}}_{2,r})}\\
+\|\nabla b\|_{\wt L^{\rho'}_t(\dot B^{s+1-\frac{2}{\rho}}_{2,r})}\|a\|_{\wt L^{\rho}_t(\dot B^{\frac{2}{\rho}-2}_{\infty, \infty})}.\end{multline}
\end{proposition}
\begin{proof}
Proving the two inequalities relies  on the decomposition 
\begin{equation}\label{decompo}
[\dot\Delta_j, b]a=[\dot\Delta_j,T_b]a+ \dot\Delta_j(T_ab+R(a,b))-(T_{\dot\Delta_ja}b+R(\dot\Delta_ja,b)).
\end{equation}
For getting \eqref{com01}, we bound the first term of \eqref{decompo}
as follows (use  \cite[Ineq. (2.58)]{Ba11}):  
\begin{align*}
\sum_{j\in\mathbb{Z}}2^{js}\|[\dot\Delta_j,T_b]a\|_{L^2}\lesssim\|\nabla b\|_{L^\infty}\|a\|_{\dot B^{s-1}_{2,1}}.
\end{align*}
The next two terms of \eqref{decompo} may be bounded by  using the fact that 
the remainder and paraproduct operator
map $\dot B^{-1}_{\infty,\infty}\times\dot B^{s+1}_{2,1}$ to $\dot B^{s}_{2,1}$.
Finally, owing to the properties of localization of the Littlewood-Paley decomposition, we have
\begin{equation}\label{decompo1}
T_{\dot\Delta_ja}b+R(\dot\Delta_ja,b)=\sum_{j'\geq j-2} \dot S_{j'+2}\dot\Delta_ja\,\dot\Delta_{j'}b.
\end{equation}
{}From  Bernstein inequality and  $\| \dot S_{j'+2} a\|_{L^\infty}\lesssim 2^{j'}\|a\|_{\dot B^{-1}_{\infty,\infty}},$ we gather 
$$\begin{aligned}
\sum_{j}2^{js} \|T_{\dot\Delta_ja}b+R(\dot\Delta_ja,b)\|_{L^2}&\lesssim
\sum_j\sum_{j'\geq j-2}2^{js}\|\dot S_{j'+2}a\|_{L^\infty}\|\dot\Delta_{j'}b\|_{L^2}\\
&\lesssim \|a\|_{\dot B^{-1}_{\infty,\infty}}\sum_j\sum_{j'\geq j-2}2^{s(j-j')}\,2^{j's}\|\nabla\dot\Delta_{j'}b\|_{L^2}\\&\lesssim \|a\|_{\dot B^{-1}_{\infty,\infty}}\|\nabla b\|_{\dot B^{s}_{2,1}}.\end{aligned}$$

To prove \eqref{com2}, we observe that owing to the localization properties of the Littlewood-Paley decomposition, the first term of \eqref{decompo} may be decomposed into
$$[\ddj,T_b]a =\sum_{|j'-j|\leq4} [\ddj, \dot S_{j'-1}b]\dot\Delta_{j'}a.$$
Now, according to \cite[Lem. 2.97]{Ba11},  we have 
$$
\| [\ddj, \dot S_{j'-1}b]\dot\Delta_{j'}a\|_{L^2}\lesssim 2^{-j}\|\nabla\dot S_{j'-1}b\|_{L^\infty}\|\dot\Delta_{j'}a\|_{L^2},
$$
and, since $\frac{2}{\rho}-1<0$, 
$$\|\nabla S_{j'-1}b\|_{L_t^\rho(L^\infty)}\lesssim 2^{j'(1-\frac2\rho)} \|\nabla b\|_{\wt L_t^\rho(\dot B^{\frac2\rho-1}_{\infty,\infty})}.$$
Hence, for all $(j,j')\in\Z^2$ such that $|j-j'|\leq 4,$
$$2^{js} \|[\ddj, \dot S_{j'-1}b]\dot\Delta_{j'}a\|_{L^1_t(L^2)}\lesssim 
2^{js} 2^{-\frac{2}{\rho}j'}\|\dot\Delta_{j'}a\|_{L^{\rho'}_t(L^2)}
\|\nabla b\|_{\wt L^{\rho}_t(\dot B^{\frac{2}{\rho}-1}_{\infty,\infty})}.$$
Therefore,  summing up on $j'\in\{j-4,j+4\},$ then 
 taking the $\ell^r(\Z)$ norm,
$$\bigl\|2^{js}\|[\ddj,T_b]a  \|_{L_t^1(L^2)}\bigr\|_{\ell^r} \lesssim  \|\nabla b\|_{\wt L^{\rho}_t(\dot B^{\frac{2}{\rho}-1}_{\infty,\infty})} \|a\|_{\wt L^{\rho'}_t(\dot B^{s-\frac{2}{\rho}}_{2,r})}.$$
The next two terms may be bounded according
to Proposition \ref{P_para} and Remark \ref{Re_3}: 
$$\bigl\|2^{js}\|\ddj T_ab \|_{L_t^1(L^2)}\|_{\ell^r} +\bigl\|2^{js}\|\ddj R(a,b)\|_{L_t^1(L^2)}\|_{\ell^r} 
\lesssim \|a\|_{\wt L^{\rho}_t(\dot B^{\frac{2}{\rho}-2}_{\infty, \infty})}
\|\nabla b\|_{\wt L_t^{\rho'}(\dot B^{s+1-\frac{2}{\rho}}_{2,r})}.$$
Finally, use  \eqref{decompo1} and the fact that 
$$\| \dot S_{j'+2} a\|_{L^{\rho}_t(L^\infty)}\lesssim 2^{(2-\frac{2}{\rho})j'}\|a\|_{\wt L^{\rho}_t(\dot B^{\frac{2}{\rho}-2}_{\infty,\infty})}$$ to get
 $$\begin{aligned}
 2^{js}\| T_{\ddj a}b+R(\ddj a,b)\|_{L^1_t(L^2)}
&\lesssim \sum_{j'\geq j-2}2^{js}\|\dot S_{j'+2}a\|_{L^{\rho}_t(L^\infty)}\|\dot\Delta_{j'}b\|_{L^{\rho'}_t(L^2)}\\
 &\lesssim  \|a\|_{\wt L_t^{\rho}(\dot{B}^{\frac{2}{\rho}-2}_{\infty, \infty})}\!\sum_{j'\geq j-2} \!\!2^{s(j-j')}\,2^{(s+2-\frac{2}{\rho})j'}\!\|\dot\Delta_{j'} b\|_{L_t^{\rho'}(L^2)}.
\end{aligned}$$
 Taking the $\ell^r(\Z)$ norm of both sides and using a convolution inequality for series
 (remember that  $s>0$),  we end up with 
 $$
 \bigl\|2^{js}\| T_{\ddj a}b+R(\ddj a,b)\|_{L^1_t(L^2)}\bigr\|_{\ell^r(\Z)} \lesssim
  \|a\|_{\wt L^{\rho}_t(\dot B^{\frac{2}{\rho}-2}_{\infty,\infty})}
\|\nabla b\|_{\wt L_t^{\rho'}(\dot B^{s+1-\frac{2}{\rho}}_{2,r})}.$$
This completes the proof of Inequality \eqref{com2}. 
\end{proof}

\bigbreak\bigbreak

\noindent\textsc{Universit\'e Paris-Est Cr\'eteil,  LAMA UMR 8050, 61 avenue du G\'en\'eral de Gaulle,  94010 Cr\'eteil
et Sorbonne Universit\'e, LJLL UMR 7598, 4 Place Jussieu, 75005 Paris}\par\nopagebreak
E-mail address: raphael.danchin@u-pec.fr
\medbreak
\noindent\textsc{Universit\'e Paris-Est Cr\'eteil,  LAMA UMR 8050, 61 avenue du G\'en\'eral de Gaulle,  94010 Cr\'eteil}\par\nopagebreak
E-mail address: jin.tan@u-pec.fr

\end{document}